\documentclass[reqno,12pt]{amsart}

\usepackage{amssymb,amsmath, latexsym,mathabx,color}
\usepackage{enumerate}
\usepackage[all]{xy}

\theoremstyle{definition}
\newtheorem{notation}{Notation}[section]

\theoremstyle{plain}

\newtheorem{lemma}[notation]{Lemma}
\newtheorem{thm}[notation]{Theorem}

\theoremstyle{definition}
\newtheorem{definition}[notation]{Definition}

\theoremstyle{definition}
\newtheorem{remark}[notation]{Remark}

\newcommand{\id}{\operatorname{id}}

\begin{document}

\title[Conditionally Free Product]{Conditionally Free reduced products of Hilbert spaces}

\author[O. Arizmendi]{Octavio Arizmendi}
\address{Centro de Investigaci\'on en Matem\'aticas,
				Jalisco s/n, Mineral de Valenciana,
				Guanajuato, Guanajuato, C.P. 36240, M\'exico}
\email{octavius@cimat.mx}

\author[M. Ballesteros]{Miguel Ballesteros}
\address{ Mathematical Physics Department,
                Instituto de Investigaciones en Matem\'aticas Aplicadas y Sistemas,
                Universidad Nacional Aut\'onoma de M\'exico, Campus  C.U.,
                Circuito Escolar 3000, C.P. 04510, M\'exico City, M\'exico}
\email{miguel.ballesteros@iimas.unam.mx}

\author[F. Torres-Ayala]{Francisco Torres-Ayala}
\address{Facultad de Ciencias, Departamento de Matem\'aticas,
				Universidad Nacional Aut\'onoma de M\'exico, Campus  C.U.,
               Circuito Exterior s/n, C.P. 04510, M\'exico City, M\'exico}
\email{tfrancisco@ciencias.unam.mx}

\thanks{Research supported by UNAM-PAPIIT, grant  IA105316}

\date{September 30th 2018}



\begin{abstract}

We present a product of pairs of pointed Hilbert spaces that, in the context of Boz\.ejko, Leinert and Speicher's theory of conditionally free probability, plays the  role   of the reduced free product of pointed Hilbert spaces, and thus gives a unified construction for the natural notions of independence defined by Muraki.

 We additionally provide important applications of this construction.  We prove that, assuming  minor restrictions,  for any pair of conditionally free algebras there are copies of them that are conditionally free and also free,  a property that is frequently assumed (as hypothesis) to prove several results in the literature. Finally, we give a short proof of the linearization property of the $^cR$-transform (the analog of Voiculescu's $R$-transform in the context of conditionally free probability).

\end{abstract}

\maketitle

\section{Introduction}

Non-commutative probability is a theory that naturally arises from the description of expectation values of observables (non commutative random variables) in quantum mechanics. It is a well-established concept that is  widely used in  communities of physicists, mathematicians and mathematical physicists, and it is well-known since the beginning of quantum mechanics. 

In the early 90's, an innovative idea was introduced by Voiculescu that created a new branch in mathematics, which is called free probability. A key contribution of Voiculescu is a definition of independence for  random variables that is the non-commutative version of independence in the commutative case, which is a central concept in probability theory. The notion of independence is necessarily attached to constructions products that provide independent (non-commutative) random variables (similar to the  Kolmogorov extension theorem for the commutative case). Since the original input of Voiculescu, several notions of independence and their corresponding (free) products were introduced, each one of them used to solve specific problems.

In this paper we present for the first time a unified version of (free) product that contains and generalizes all notions of (free) products so far introduced  in the context of independence for non-commutative random variables, see \cite{Mur2}.  Our construction can be viewed as an analog  of the product associated to the theory of conditionally free probability by Boz\.ejko, Leinert and Speicher, for pointed Hilbert spaces.  

The notion of conditionally free independence was introduced first by Bo\.zejko  in the context of harmonic analysis in \cite{Boz}, and later by Boz\.ejko and Speicher in \cite{BS}  and  \cite{BLS}, using  $C^*$-algebras. It is called $\psi$-independence  in  \cite{BS} and it is defined as follows (see  \cite{BLS}):

\begin{definition}\label{Def1}
Let $(\mathcal{A}, \varphi, \psi)$  be a unital algebra  with two distinguished states, $   \varphi $   and $ \psi $. A family of unital sub-algebras $\{\mathcal{A}_\iota\}_{\iota \in F}$ is called conditionally free (c. free  for short), with respect to $(\varphi, \psi)$,  if for any natural $n\geq 2$, indices
$\kappa_i\in F $ with $\kappa_1 \neq \cdots \neq \kappa_n$ and elements $a_i\in \mathcal{A}_{\kappa_i}$, the relation  $\psi(a_i)=0$, for all $i$, implies $\varphi(a_1\cdots a_n)=\varphi(a_1)\cdots \varphi(a_n)$.
\end{definition}

A product associated to the concept  ``c. free'' is provided in \cite{BS} in the following way:

\begin{definition}
Let $F$ be a set of indices.  For every $\iota $ in $F$, let $(\mathcal{A}_\iota, \varphi_\iota, \psi_\iota )$  be a unital $C^*$-algebra with two states, $\varphi_\iota$ and $ \psi_\iota $. A pair $(\hat{\mathcal{A}}, \hat{\varphi})$, where $\hat{\mathcal{A}}$ is a unital $C^*$-algebra and $\hat{\varphi}$ is a state on $\hat{\mathcal{A}}$, is called a reduced $\psi$-product for $\{(\mathcal{A}_\iota, \varphi_\iota, \psi_\iota)\}_{\iota \in F}$ if the following properties hold:
\begin{enumerate}
\item there exist unital $*$-homomorphisms $j_\iota: \mathcal{A}_\iota \to \hat{\mathcal{A}}$ such that 
$\hat{\mathcal{A}}$ is generated by $\cup_{\iota \in F}j_\iota(\mathcal{A}_\iota)$;
\item $\hat{\varphi}\circ j_\iota=\varphi_\iota$ for all $\iota$ in $F$;
\item  for every natural number $n \geq 2$, every set of  indices  $\{\kappa_1, \cdots, \kappa_n  \}\subset F$ with $\kappa_1 \neq \cdots \neq \kappa_n$ and  every  $a_i \in \mathcal{A}_{\kappa_i}$,  $i \in \{ 1, \cdots n \}$, such that $\psi_{\kappa_i}(a_i)=0$, for every $i \in \{ 1, \cdots n \} $,  we have

\begin{align}\label{fact}
\hat{\varphi}(j_{\kappa_1}(a_1) \cdots j_{\kappa_n}(a_n))=\varphi_{\kappa_1}(a_1)\cdots \varphi_{\kappa_n}(a_n).
\end{align} 

\end{enumerate}
\end{definition}

The existence of $\psi$-products is proved in \cite{BS} under general conditions:

\begin{thm}\cite{BS}\label{Ex}
For each family $\{(\mathcal{A}_\iota, \varphi_\iota, \psi_\iota)\}_{\iota \in F}$  of unital $C^*$-algebras with two states, there exists a reduced $\psi$-product $(\hat{\mathcal{A}},\hat{\varphi})$.
\end{thm}

The proof of Theorem \ref{Ex}  in \cite{BS} uses  GNS representations with respect to the sates $\varphi_{\iota}$ and $\psi_{\iota}$, for every $\iota \in F$. They choose particular representations in such a way  that for every $\iota \in F$ the Hilbert spaces associated to $\varphi_{\iota}$ and $\psi_{\iota}$ coincide.    

In this  work we present a similar construction but instead of taking the same space for both states we take  \emph{different} spaces for each state, $H_\iota$ for $\varphi_\iota$ and $K_\iota$ for $\psi_\iota$ (see Theorem \ref{Coro:PsiProducts} below).  While this might seem to be a minor modification, having different spaces gives more flexibility and it leads to a product, that we call the conditionally free product of \emph{pairs} of Hilbert spaces, along with natural embeddings. 

An advantage  of the conditionally free product that we present in this paper  is that it contains (and generalizes) all products so far introduced in the context of independent random variables in the non-commutative  case. It naturally interpolates the constructions for products in the free, Boolean, monotone and orthogonal notions of independence. Each one of these notions of independence has a natural product of Hilbert spaces associated to them. For free independce the product is the reduced free product emerged from the works of Avitzur in \cite{Avi} and  Voiculescu in \cite{Voi1};  for the Boolean and monotone notions of  independence,  it is usually the tensor product; lastly, more recently, Lenczewski introduced the notion of orthogonal independence in \cite{Len}, along with an orthogonal product of Hilbert spaces. The product that we propose completes this picture and unifies all these constructions.

An important application of our construction is the following: It is assumed in several sources (see \cite{PW} and \cite{Pop}, for example) that conditionally free algebras with respect to the states $\varphi$ and $\psi$, see Definition \ref{Def1}, are also free with respect to $\psi$ - see Definition \ref{Voicu}.  Although a general characterization of when c. free algebras are also free is still an open question, in this paper we give a satisfactory answer in the following sense: given c. free algebras we can always find, under mild assumptions, copies of them  that are conditionally free and also free (in the way it is specified above), see Theorem \ref{Thm:FreeCopiesCFAlgebras}.

This paper is divided in five sections. In 
Section \ref{Section2} we provide basic definitions of the various types of independence. In Section \ref{Section3} we construct the conditionally free product of pairs of pointed Hilbert spaces that we introduce in this paper,  along with embeddings of operators. In Section \ref{Section4}  we prove  Theorem \ref{Thm:FreeCopiesCFAlgebras} (we explain above the content of this theorem)   and in Section \ref{SEction5} we  prove the linearization property of  $^cR$. 

\section{Basic definitions  and notations}\label{Section2}

Two Hilbert spaces, $H$ and $K$, are isomorphic, denoted $H\simeq K$, if there is a unitary operator $U:H\to K$.

For a Hilbert space $H$, $\mathbb{B}(H)$ denotes the bounded operators on $H$.

By a pointed Hilbert space we mean a pair $(H,\xi)$, where $H$ is a complex Hilbert space and $\xi$ is a unit vector in $H$. Given a pointed vector space $(H,\xi)$ we denote
$$
H^o:=H\ominus \mathbb{C}\xi=\{h\in H: \langle h, \xi \rangle_H=0 \},
$$
and we let $P_\xi, P_\xi^\perp : H\to H $ stand for the orthogonal projections onto $\mathbb{C}\xi$ and $H^o$, respectively. Hence $P_\xi^\perp = \id_H-P_\xi$. We frequently denote the elements in $H^o$ by $h^o$.

For a unital algebra $\mathcal{A}$, a state of $\mathcal{A}$ is a linear functional $\varphi:\mathcal{A} \to \mathbb{C}$ that satisfies $\varphi(1_\mathcal{A})=1$. If in addition $\mathcal{A}$ is a $*$-algebra we request
 that $\varphi(a^*a)\geq 0$ for any $a$. A non commutative probability space is a pair $(\mathcal{A}, \varphi)$ where $\mathcal{A}$ is an algebra and $\varphi$ is a state of $\mathcal{A}$.

\subsection{Free independence}

\begin{definition}[Voiculescu, \cite{Voi1}] \label{Voicu}
Let $(\mathcal{A},\varphi)$ denote a non commutative probability space. A pair of unital sub-algebras $\mathcal{A}_\alpha$ and $\mathcal{A}_\beta$ are called freely independent, with respect to $\varphi$ if, given a natural  $n\geq 2$, indices  $\kappa_1\neq\cdots \neq \kappa_n$  in $  \{ \alpha, \beta \}$, and elements $a_i\in \mathcal{A}_{\kappa_i}$ with $\varphi(a_i)=0$, for all $i$, it holds that $\varphi(a_1 \cdots a_n)=0$.
\end{definition}

Now we recall Voiculescu's construction
for reduced free product of Hilbert spaces and regular left representations.

Given two pointed Hilbert spaces, $(H_\alpha, \xi_\alpha)$ and $(H_\beta,\xi_\beta)$, its reduced free product, denoted $(H_\alpha, \xi_\alpha)*(H_\beta, \xi_\beta)=(H,\xi)$, is given by 
$$
H :=\mathbb{C}\xi \oplus \bigoplus_{n\geq 1} \bigoplus_{i_1 \neq  \cdots \neq i_n} H_{i_1}^o\otimes \cdots \otimes    H_{i_n}^o .
$$
Here we choose the symbol $\xi$ to denote a distinguished  element in $H$, with norm $1$.

We also have embeddings of $\mathbb{B}(H_\iota)$ into $\mathbb{B}(H)$ given as follows:  for fixed $\iota$ let
$$
H(\iota):=\mathbb{C}\xi \bigoplus_{n\geq 1} \bigoplus_{\substack{i_1\neq \cdots \neq i_n \\ i_1\neq\iota}}H_{i_1}^o\otimes \cdots \otimes  H_{i_n}^o    .
$$

We observe that   $H_\iota \otimes H(\iota) \simeq H$. This isomorphism is given by the unitary operator $V_\iota : H \to  H_\iota \otimes H(\iota)$, defined in the following way: 
$$V_\iota(\xi)=\xi_\iota \otimes \xi,$$ and for $h_j^o \in H_{\kappa_j}^o$, $\kappa_1\neq \cdots \neq\kappa_n$,
$$
V_\iota(h_1^o\otimes \cdots \otimes h_n^o)=\left\{ \begin{array}{cc}
h_1^o\otimes    (h_2^o \otimes \cdots h_n^o ) , & \textrm{if   $  \kappa_1=\iota $  and $n\geq 2$,}\\
 h_1^o    \otimes \xi, & \textrm{if    $  \kappa_1=\iota $ and $n=1$,  }\\
\xi_\iota \otimes  
 (h_1^o \otimes \cdots \otimes h_n^o)  , & \textrm{if $\kappa_1\neq\iota$.}
\end{array}
\right.
$$

We define an embedding from $  \mathbb{B}(H_\iota)$ into $\mathbb{B}(H)$ by 
$$T\mapsto \lambda_\iota(T):=V_\iota^{-1}(T\otimes \id_{H(\iota)} )V_\iota.$$

It is well known (see \cite{VDN}) that the algebras $\lambda_\alpha(\mathbb{B}(H_\alpha))$ and $\lambda_{\beta}(\mathbb{B}(H_\beta))$ are freely independent, with respect to the vector state at $\xi$.

\subsection{Boolean independence}

\begin{definition}[Speicher, Woroudi,  \cite{SW}]
Let $(\mathcal{A}, \varphi)$ denote a non commutative probability space. A pair of non-unital sub-algebras of $\mathcal{A}$, $\mathcal{A}_\alpha$ and $\mathcal{A}_\beta$, are Boolean independent, with respect to $\varphi$ if, whenever there is a natural number  $n\geq 2$, indices $\kappa_1\neq \cdots \neq \kappa_n$   in $  \{ \alpha, \beta \}$, and elements $a_i\in \mathcal{A}_{\kappa_i}$, it holds that
$$
\varphi(a_1\cdots a_n)=\varphi(a_1)\cdots \varphi(a_n).
$$
\end{definition}

In Boolean independence, it is common to use the tensor product as the base space to construct Boolean products (see \cite{HO}). But one can take a smaller space to achieve the same result. 

Given a pair of pointed Hilbert spaces,  $(H_\alpha, \xi_\alpha)$ and $(H_\beta, \xi_\beta)$, we define its Boolean product as
$$
(H,\xi)=(H_\alpha, \xi_\alpha)\uplus (H_\beta, \xi_\beta),
$$
where
$$H=:H_\beta^o \oplus \mathbb{C}\xi \oplus H_\alpha^o .$$
Here we choose the symbol $\xi$ to denote a distinguished element in $H$, with norm $1$.

We set  the maps $V_\iota:H \to H_\iota$ given by 
$$
V_\iota(h)=\left\{
\begin{array}{cc}
 	\xi_\iota, & \textrm{if $h=\xi$} \\
 	h, & \textrm{if $h\in H_\iota^o$}\\
 	0,   & \textrm{otherwise.}
 	\end{array}
\right.
$$

Using   these maps we define  the embeddings from $\mathbb{B}(H_\iota)$ into $\mathbb{B}(H)$ by
$$
T \mapsto V_\iota^* T V_\iota.
$$

Then the sub-algebras
$$\{V_\alpha^* T_\alpha V_\alpha: T_\alpha \in \mathbb{B}(H_\alpha)\}, \{V_\beta^* T_\beta V_\beta: T_\beta \in \mathbb{B}(H_\beta)\},$$
 are non-unital and  Boolean independent with respect to the vector state at $\xi$.

 \subsection{Monotone independence}
 
 In this type of independence we have to assume that the set of indexes $F$ is totally ordered. 
 
 For a sequence of indices
 $$
 \kappa_1 \neq \cdots \neq \kappa_n
 $$
with $n\geq 2$, we say that $\kappa_p$ is a peak in the sequence if either
\begin{enumerate}
\item $1< p< n$, $\kappa_{p-1}<\kappa_p$ and $\kappa_p >\kappa_{p+1}$, or;
\item $p=1$ and $\kappa_1>\kappa_2$, or;
\item $p=n$ and $\kappa_{n-1} <\kappa_n$.
\end{enumerate}

\begin{definition}[Muraki, \cite{Mur}]
Two sub-algebras $\mathcal{A}_\alpha$ and $\mathcal{A}_\beta$ (where we are assuming the ordering $\alpha <\beta$) are  called monotone independent, with respect to $\varphi$ if, given $n\geq 2$, indices $\kappa_1 \neq \cdots \neq \kappa_n$   in $  \{ \alpha, \beta \}$, and elements $a_i \in \mathcal{A}_{\kappa_i}$, it holds that
$$
\varphi(a_1 \cdots a_n)=\varphi(a_p)\varphi(a_1 \cdots \check{a}_p \cdots a_n),
$$
where by $\check{a}_p$ we mean that the element is omitted and $p$ is the first peak in the sequence $\kappa_1\neq \cdots \neq \kappa_n$.
\end{definition}

For a pair of pointed Hilbert spaces $(H_\alpha, \xi_\alpha)$ and $(H_\beta, \xi_\beta)$,  we define its monotone product by  $(H, \xi)$, where
$$
H :=  \Big (  H_\beta^o\otimes H_\alpha^o \Big )    \oplus H_\alpha^o \oplus \mathbb{C}\xi \oplus H_\beta^o,
$$
and we choose the symbol $\xi$ to denote a distinguished element in $H$, with norm $1$. Notice that $H$  is unitarily equivalent to $ H_\beta \otimes H_\alpha$.

 For every $\iota \in \{ \alpha, \beta \}  $ we set the maps $ V_{\iota}  :  H  \to H_{\iota}$, given by 
$$
V_\iota(h)=\left\{
\begin{array}{cc}
 	\xi_\iota, & \textrm{if $h=\xi$} \\
 	h, & \textrm{if $h\in H_\iota^o$}\\
 	0,   & \textrm{otherwise.}
 	\end{array}
\right.
$$
In addition, we define the map $V_\beta^{(0)}: H \to H_\beta \otimes H_\alpha^o $ as follows
$$
V_\beta^{(0)}(h)=\left\{
\begin{array}{cc}
 	\xi_\beta \otimes h, & \textrm{if $h\in H_\alpha^o$} \\
 	h, & \textrm{if $h\in H_\beta^o \otimes  H_\alpha^o$}\\
 	0,   & \textrm{otherwise.}
 	\end{array}
\right.
$$
 
The maps above induce  the embeddings 
$$
T_\alpha\in \mathbb{B}(H_\alpha) \mapsto V_\alpha^*T_\alpha V_\alpha, \quad T_\beta \in \mathbb{B}(H_\beta) \mapsto V_\beta^* T_\beta V_\beta + V_\beta^{(0)*}(T_\beta\otimes \id_{H_\alpha^o}) V_\beta^{(0)},
$$
and the sub-algebras
 $$ \{V_\alpha^* T_\alpha V_\alpha : T_\alpha \in \mathbb{B}(H_\alpha)\}, \{V_\beta^* T_\beta V_\beta + V_\beta^{(0)*}(T_\beta \otimes \id_{H_\alpha^o}) V_\beta^{(0)} : T_\beta \in \mathbb{B}(T_\beta)\} $$ 
 are  non-unital and monotone independent with respect to the vector state at $\xi$.

\subsection{Orthogonal independence}

\begin{definition}[Lenczewski, \cite{Len}]
Let $(\mathcal{A}, \varphi, \psi)$	denote a unital algebra with a pair of states and let $\mathcal{A}_\alpha$, $\mathcal{A}_\beta$ be two non-unital sub-algebras of $\mathcal{A}$. We say that $\mathcal{A}_\beta$ is orthogonal to $\mathcal{A}_\alpha $, with respect to $(\varphi, \psi)$ if 
\begin{enumerate}
\item $\varphi(ba_2)=\varphi(a_1b)=0$;
\item $\varphi(w_1a_1ba_2w_2)=\psi(b)(\varphi(w_1a_1a_2w_2)-\varphi(w_1a_1)\varphi(a_2w_2))$;
\end{enumerate}
for any $a_1,a_2$ elements in $\mathcal{A}_\alpha$, $b$ in $\mathcal{A}_\beta$ and $w_1$, $w_2$ elements in the algebra generated by $\mathcal{A}_\alpha$ and $\mathcal{A}_\beta$. 
\end{definition}

Given a pair of pointed Hilbert spaces, $(H_\alpha, \xi_\alpha)$, $(H_\beta, \xi_\beta)$, the orthogonal product $(H,\xi)=(H_\alpha, \xi_\alpha)\vdash (H_\beta, \xi_\beta)$ is defined by
$$
H=\mathbb{C}\xi\oplus H_\alpha^o \oplus \Big( H_\beta^o\otimes H_\alpha^o \Big). 
$$
Here we choose the symbol $\xi$ to denote a distinguished element in $H$, with norm $1$.

In this case we set the embeddings 
\begin{eqnarray*}
T\in \mathbb{B}(H_\alpha) \mapsto \tau_\alpha(T):=U^*(T\otimes P_{\xi_\beta})U \in \mathbb{B}(H),\\
S\in \mathbb{B}(H_\beta) \mapsto \tau_\beta(S):=U^*(P_{\xi_\alpha}^\perp \otimes S)U\in \mathbb{B}(H),
\end{eqnarray*}
where $U:H \to H_\alpha \otimes H_\beta$ is the isometry given by
$$
U(\xi)=\xi_\alpha \otimes \xi_\beta, U(h_\alpha^o)=h_\alpha^o\otimes \xi_\beta, U(h_\beta^o\otimes h_\alpha^o)=h_\alpha^o\otimes h_\beta^o.
$$

In \cite{Len} Lenczewski proved the next result.

\begin{thm}
Let $(H,\xi)=(H_\alpha, \xi_\alpha)\vdash (H_\beta, \xi_\beta)$, and let $\varphi$, $\varphi_\alpha$ and $\varphi_\beta$ be the states associated to $\xi,    \xi_\alpha      $ and $\xi_\beta$, respectively. We take  a unit vector $\eta \in H^o$ and    let $\psi$ denote the state of $\mathbb{B}(H)$  associated with it. Then 
\begin{enumerate}
\item $\tau_\beta(\mathbb{B}(H_\beta))$ is orthogonal to $\tau_\alpha(\mathbb{B}(H_\alpha))$ with respect to $(\varphi, \psi)$;
\item $\varphi \circ \tau_\alpha$ agrees with $\varphi_\alpha$ on $\mathbb{B}(H\alpha)$;
\item $\psi \circ \tau_\beta$ agrees with $\varphi_\beta $ on $\mathbb{B}(H_\beta)$.
\end{enumerate}

\end{thm}

\section{Conditionally Free Products of pointed Hilbert Spaces} \label{Section3}

\subsection{Construction}

In this section we present the construction of what we call the conditionally free product of pairs of pointed Hilbert spaces. The construction can be done for any finite number of pairs of pointed Hilbert spaces but for sake of clarity we write all definitions and proofs for only two pairs.

We start with some notation.

Given a natural number $n\geq 1$ and a   set $F$,  a function from $\{1,\dots, n\}$ to $F$ will be denoted  by 
$\underline{i}=(i_1, \dots, i_n)$, implying that $m$ is mapped to $i_m$ under $\underline{i}$. The empty function will be denoted as $\underline{\emptyset}$. We write $F^{[\mathbb{N}]}$ for the set of all functions of
the form $\underline{i}=(i_1,\dots, i_n)$, $n\geq 1$. For a function $\underline{i}$ in $F^{[\mathbb{N}]} \cup \{ \underline{\emptyset}\}$, we defined its length  as $|\underline{i} |=n$, with the convention that $|\underline{\emptyset}|=0$.

\begin{definition}
Given two pairs of pointed Hilbert spaces, $((H_i,\xi_i),(K_i,\eta_i))$, $i\in \{\alpha, \beta\}$, their conditionally free product (abbreviated c. free product) is defined as the pair
$$
((H_\alpha,\xi_\alpha),(K_\alpha,\eta_\alpha))*((H_\beta,\xi_\beta),(K_\beta,\eta_\beta)):=((H,\xi), (K,\eta)),
$$
where:
\begin{enumerate}
\item $(K,\eta)$ is Voiculescu's   free product space $(K_\alpha,\eta_\alpha)*(K_\beta,\eta_\beta)$.
\item Regarding $(H,\xi)$, 
\begin{eqnarray*}
H:= & & \big(     \bigoplus_{m=1}^\infty    \bigoplus_{ j_1\neq j_2\neq \cdots \neq j_m \neq\alpha }   K_{j_1}^o \otimes K_{j_2}^o\otimes \cdots \otimes     K_{j_m}^o   \otimes H_\alpha^o \big)  \\
 &\oplus& H_\alpha^o \oplus  \mathbb{C}\xi \oplus H_\beta ^o  \\
  &\oplus& \big( \bigoplus_{n=1}^\infty \bigoplus_{ i_1\neq i_2\neq \cdots \neq i_n \neq\beta  }  K_{i_1}^o \otimes K_{i_2}^o\otimes \cdots \otimes  K_{i_n}^o   \otimes H_\beta^o \big),
\end{eqnarray*}
where $\xi$ is a distinguished unit vector.
\end{enumerate}
\end{definition}

\begin{notation}
Let $F=\{\alpha, \beta\}$ and $\iota$ in $F$. We define
$$
I(\iota):=\{\underline{i}=(i_1,\dots, i_n)\in F^{[\mathbb{N}]}:   i_1\neq i_2\neq \cdots i_{n-1}\neq i_n\neq\iota  \} \cup \{ \emptyset \}.
$$

Given $\underline{i}=(i_1,\dots, i_n)\in I(\iota)$,  with $n\geq 1$,  we set:
$$
K^o_{\underline{i}}:=K_{i_1}^o\otimes \cdots \otimes K_{i_n}^o
$$
and
$$
K_{\underline{\emptyset}}^o=\mathbb{C}.
$$
We also take the convention that
$$
K_{\emptyset}^o\otimes L=L,
$$
for any Hilbert space $L$.

With this notation we can rewrite the first component of the c. free    
 product of Hilbert spaces as:
$$
H=\bigg( \bigoplus_{\underline{j}\in I(\alpha)} K_{\underline{j}}^o \otimes H_\alpha ^o  \bigg)  \oplus \mathbb{C}\xi  \oplus  \bigg(\bigoplus_{\underline{i}\in I(\beta)} K_{\underline{i}}^o\otimes H_\beta ^o \bigg)
$$
\end{notation}

\begin{remark} 
The   c. free product of Hilbert spaces  contains the free, Boolean,  monotone and orthogonal products of Hilbert spaces.
\begin{enumerate}
\item For the free product, set $H_i=K_i$, $\xi_i=\eta_i$, $i\in \{ \alpha, \beta\}$, to obtain
$$
H=\bigg( \bigoplus_{\underline{j}\in I(\alpha)} H_{\underline{j}}^o \otimes H_\alpha ^o  \bigg)   \oplus \mathbb{C}\xi \oplus  \bigg(\bigoplus_{\underline{i}\in I(\beta)} H_{\underline{i}}^o\otimes H_\beta^o \bigg),
$$
which is the usual free product  $(H_\alpha,\xi_\alpha)*(H_\beta,\xi_\beta)$.

\item For the Boolean product, set $K_i=\mathbb{C}$, $\eta_i=1$, $i\in\{ \alpha, \beta\}$. In this situation, $K_{\underline{i}}^o=\{0\}$, for any $\underline{i}\in I(\iota)$, wit $|\underline{i}| \geq 1$, $\iota\in\{ \alpha, \beta\}$, so that
$$
H=H_\alpha ^o \oplus \mathbb{C}\xi \oplus H_\beta ^o ,
$$
which equals $(H_1,\xi_\alpha )\uplus(H_2,\xi_\beta)$.

\item For the monotone product, set $K_\alpha =\mathbb{C}, \eta_\alpha=1$, $H_\beta=K_\beta, \xi_\beta=\eta_\beta$. In this case:
$K_{\underline{i}}^o=\{0\}$ for any $\underline{i}\in I(\beta) $, with $|\underline{i}|\geq 1$; $K_{\underline{j}}=\{0\}$ for any
$\underline{j}\in I(\alpha)$ with $|\underline{j}| \geq 2$; $K_\beta^o=H_\beta^o$.  Hence
$$
H=H_\beta^o \otimes H_\alpha^o \oplus H_\alpha^o \oplus  \mathbb{C}\xi \oplus H_\beta^o,
$$
which is, in a natural way,  unitarily isomorphic to $H_\alpha \otimes H_\beta$.

\item For  Lenczewski's orthogonal product, set $K_\alpha=\mathbb{C}, \eta_\alpha=1$, $H_\beta=\mathbb{C}, \eta_\beta=1$. In this case:
$K_{\underline{j}}^o=\{0\}$ for any $\underline{j}\in I(\alpha)$, with $|j|\geq 2$; $H_\beta^o=\{0\}$. Hence
$$
H=K_\beta^o\otimes H_\alpha^o\oplus H_\alpha^o \oplus \mathbb{C}\xi,
$$
which, in Lenczewski's notation, is $(H_\alpha,\xi_\alpha) \vdash (K_\beta,\eta_\beta)$.

\end{enumerate}

\end{remark}

\subsection{Embeddings}

Now that we have introduced the c. free product of Hilbert spaces, the next step is to find a way to embed operators into the bounded operators of the  c. free product, in such a way that the embedding  preserves the original states. In order to accomplish this we  write $H$ in different ways, according to which pair of operators, either from  $\mathbb{B}(H_\alpha)\times \mathbb{B}(K_\alpha)$ or $\mathbb{B}(H_\beta) \times \mathbb{B}(K_\beta)$, we want to embed.

Firstly, let us denote
\begin{eqnarray*}
I(\kappa ; \iota)& : =&\{\underline{i}=(i_1,\dots, i_n) \in I(\iota): i_1=\kappa \}, \quad \kappa \in \{ \alpha, \beta\},  
\end{eqnarray*}
and, for the rest of the discussion, we use the notation $\overline{\alpha}=\beta$, $\overline{\beta}=\alpha$.

For instance, the element in $I(\iota)$ of  length 1,  namely $(\overline{\iota})$,  belongs to $I(\overline{\iota};\iota)$. The element in $I(\iota)$ of length 0, namely $\underline{\emptyset}$, belongs to $I(\iota;\iota)$.

Notice that we can  express  $I(\iota)$ as 
$$
I(\iota)= \{\emptyset\} \sqcup  I(\overline{\iota};\iota) \sqcup \tilde{I}(\overline{\iota};\iota),
$$
where 
$$
\tilde{I}(\overline{\iota};\iota): = \{ \tilde{\underline{i}}\in I(\iota): \tilde{\underline{i}}=(\iota, \underline{i}), \underline{i}\in I(\overline{\iota};\iota)  \},
$$
and by $\sqcup$ we denote the disjoint union.

On the other hand, $I(\iota)$ can also be rewritten as
$$
I(\iota)=  I({\iota};\iota) \sqcup   \check{I}(\iota;\iota),
$$
where
\begin{eqnarray*}
\check{I}(\iota;\iota)& : =&  \{  
     \underline{\check{i}}      \in I(\iota):             \underline{\check{i}}        =(\overline{\iota},\underline{i} ), \underline{i}\in I(\iota;\iota) \},
\end{eqnarray*}
and we take  the convention that $(\overline{\iota}, \underline{\emptyset})=(\overline{\iota})$.

With this notation  we can rewrite the c. free  product of Hilbert spaces in the following two forms:
\begin{align*}
H=&\bigg( \bigoplus_{\underline{j}\in I(\beta;\alpha)} K_{\underline{j}}^o \otimes H_\alpha^o \oplus K_\alpha ^o\otimes K_{\underline{j}}^o\otimes H_\alpha^o  \bigg)   \oplus \mathbb{C}\xi \oplus H_\alpha^o    \nonumber \\
&\oplus   \bigg(\bigoplus_{\underline{i}\in I(\beta;\beta)} K_{\underline{i}}^o\otimes H_\beta^o \oplus K_{\alpha}^o\otimes K_{\underline{i}}^o\otimes H_\beta^o \bigg)   \\
=&\bigg( \bigoplus_{\underline{j}\in I(\alpha;\alpha)} K_{\underline{j}}^o \otimes H_\alpha^o \oplus K_\beta^o\otimes K_{\underline{j}}^o\otimes H_\alpha^o  \bigg)   \nonumber \\
&\oplus \mathbb{C}\xi \oplus H_\beta^o \oplus  \bigg(\bigoplus_{\underline{i}\in I(\alpha,\beta)} K_{\underline{i}}^o\otimes H_\beta^o \oplus K_{\beta}^o\otimes K_{\underline{i}}^o\otimes H_\beta^o \bigg),
\end{align*}

which, for general $\iota \in \{\alpha,\beta\}$, can be written as

\begin{eqnarray}
H=& &\bigg( \bigoplus_{\underline{j}\in I(\overline{\iota};\iota)} K_{\underline{j}}^o \otimes H_\iota ^o \oplus K_\iota ^o\otimes K_{\underline{j}}^o\otimes H_\iota^o  \bigg)  \nonumber  \\
&  &\oplus \mathbb{C}\xi \oplus H_\iota^o    \nonumber \\
& &\oplus   \bigg(\bigoplus_{\underline{i}\in I(\overline{\iota};\overline{\iota})} K_{\underline{i}}^o\otimes H_{\overline{\iota}}^o \oplus K_{\iota}^o\otimes K_{\underline{i}}^o\otimes H_{\underline{\iota}}^o \bigg) . \label{Eqn:Embedding1}
\end{eqnarray}

\begin{notation}
With the purpose of simplifying notation,  for a given $\iota$  in $\{\alpha, \beta\}$, operators in $\mathbb{B}(H_\iota)$ and $\mathbb{B}(K_\iota)$ will be denoted by $T_\iota$ and $S_\iota$, respectively.
\end{notation}

Now we proceed to the constructions of the embeddings.
Given an operator $T_\iota$ in $\mathbb{B}(H_\iota)$ we let 
\begin{eqnarray*}
T_\iota^{(\underline{\emptyset})}: & & \mathbb{C}\xi \oplus H_\iota^o \to \mathbb{C}\xi \oplus H_\iota^o  \\
& &\xi \mapsto \langle T_\iota \xi_\iota , \xi_\iota \rangle \xi + P_{\xi_\iota}^\perp T_\iota(\xi_\iota) \\
& & h^o \mapsto \langle T_\iota  h^o, \xi_\iota \rangle \xi + P_{\xi_\iota}^\perp T_\iota( h^o ), \hspace{.5cm} \text {for every $h^o \in 
 H_\iota^o $ } .
\end{eqnarray*}
In other words, $T_\iota^{(\underline{\emptyset})}$ is the operator that makes the following diagram commutative 

\[
\xymatrix{
\mathbb{C}\xi \oplus H_\iota^o  \ar@{-->}[r]^{T_\iota^{(\underline{\emptyset})}}  \ar[d]^{U_{\underline{\emptyset}}^{(\iota)}} &  \mathbb{C}\xi \oplus H_\iota^o   \\
H_\iota  \ar[r]^{T_\iota}  &  H_\iota\ar[u]^{(U_{\underline{\emptyset}}^{(\iota)})^*} 
}
\]
where $U_{\underline{\emptyset}}^{(\iota)}$ is the unitary operator that implements the natural isomorphism from $ \mathbb{C}\xi \oplus H_\iota^o $  onto  $ H_\iota$ , i.e.  $ U_{\underline{\emptyset}}^{(\iota)}  $ is unitary and
$ U_{\underline{\emptyset}}^{(\iota)} \xi  = \xi_\iota  $  and $  U_{\underline{\emptyset}}^{(\iota)} h = h  $, for every $h \in  H_\iota^o$.

Given an operator $S_\iota$ in $\mathbb{B}(K_\iota)$, we use 
Eq. \eqref{Eqn:Embedding1}  to define multiple copies of $S_\iota$. Fix $\kappa$  an element in  $\{\alpha, \beta\}$.  For every  $\underline{l} \in $  $I(\overline{\iota}; \kappa)$, we define: 

\begin{eqnarray*}
S_\iota^{(\underline{l})}: & &  K_{\underline{l}}^o\otimes H_\kappa^o \oplus  K_\iota^o\otimes K_{\underline{l}}^o\otimes H_\kappa^o    \to K_{\underline{l}}^o\otimes H_\kappa^o \oplus  K_\iota^o\otimes K_{\underline{l}}^o\otimes H_\kappa^o    \\
& & k_{\underline{l}}^o \otimes h_\kappa^o  \mapsto  \langle S_\iota \eta_\iota , \eta_\iota \rangle k_{\underline{l}}^o\otimes h_\kappa^o   +  (P_{\eta_\iota}^\perp S_\iota \eta_\iota)\otimes k_{\underline{l}}^o \otimes h_\kappa^o \\
& & k_\iota^o\otimes k_{\underline{l}}^o \otimes h_\kappa^o \mapsto \langle S_\iota k_\iota^o, \eta_\iota \rangle   k_{\underline{l}}^o   \otimes h_\kappa^o + (P_{\eta_\iota}^\perp S_\iota k_\iota^o)\otimes k_{\underline{l}}^o \otimes h_\iota^o.
\end{eqnarray*}
In other words, $S_\iota^{(\underline{l})}$ is the operator that makes the following diagram  commutative
\[
\xymatrix{
K_{\underline{l}}^o\otimes H_\kappa ^o\oplus K_\iota^o\otimes K_{\underline{l}}^o\otimes  H_\kappa^o   \ar@{-->}[r]^{S_\iota^{(\underline{l})}}  \ar[d]^{U_{\underline{l}}} & K_{\underline{l}}^o\otimes H_\kappa ^o\oplus K_\iota^o\otimes K_{\underline{l}}^o\otimes  H_\kappa^o  \\
K_\iota\otimes K_{\underline{l}}^o \otimes  H_\kappa^o \ar[r]^{S_\iota\otimes\textrm{Id}_{K_{\underline{l}}^o\otimes H_\kappa^o} }  &   K_\iota \otimes K_{\underline{l}}^o \otimes  H_\kappa^o \ar[u]^{U^*_{\underline{l}}},
}
\]
where $U_{\underline{l}}$ is the unitary operator that implements the natural isomorphism
\begin{eqnarray*}
K_{\underline{l}}^o\otimes H_\kappa ^o\oplus K_\iota^o\otimes K_{\underline{l}}^o\otimes  H_\kappa^o & \simeq & \mathbb{C}\eta_\iota \otimes K_{\underline{l}}^o\otimes H_\kappa ^o\oplus K_\iota^o\otimes K_{\underline{l}}^o\otimes  H_\kappa^o \\
&\simeq & ( \mathbb{C}\eta_\iota \oplus  K_\iota^o )\otimes K_{\underline{l}}^o\otimes  H_\kappa^o \\
&\simeq & K_\iota \otimes K_{\underline{l}}^o\otimes  H_\kappa^o.
\end{eqnarray*}

\begin{definition}\label{Def:CFEmbeddings}
Given  a pair $(T_\iota, S_\iota)$ in $\mathbb{B}(H_\iota) \times  \mathbb{B}(K_\iota)$,  we use Eq. \eqref{Eqn:Embedding1}  to define $\Lambda_{(T_\iota,S_\iota)}\in \mathbb{B}(H)$ by
$$
\Lambda_{(T_\iota,S_\iota)}:=\bigoplus_{\underline{j}\in I(\overline{\iota}, \iota)}S_\iota^{(\underline{j})} \oplus T_\iota^{(\underline{\emptyset})} \oplus  \bigoplus_{\underline{i}\in I(\overline{\iota}, \overline{\iota})} S_\iota^{(\underline{i})}.
$$

\end{definition}

Notice that each $S_\iota^{(\underline{l})}$ or $T_\iota^{(\emptyset)}$ are just copies of $S_\iota$ and $T_\iota$, respectively. Hence  we can think the operator $\Lambda_{(T_\iota, S_\iota)}$  as acting on the c. free product of Hilbert spaces in the following  way:
$$
\xymatrix @-2pc{ 
\cdots  & K_\alpha ^o\otimes K_\beta ^o \otimes H_\alpha^o \ar@/_1pc/@{<->}[rr]_{S_\alpha} & \oplus &  K_\beta^o \otimes H_\alpha^o \ar@/^1pc/@{<->}[rr]^{S_\beta}  & \oplus &  H_\alpha^o \ar@/_1pc/@{<->}[rr]_{T_\alpha} & \oplus & \mathbb{C}\xi \ar@/^1pc/@{<->}[rr]^{T_\beta} & \oplus &  H_\beta^o \ar@/_1pc/@{<->}[rr]_{S_\alpha} & \oplus  & K_\alpha^o \otimes H_\beta^o  & \cdots 
}
$$

In the diagram above, each arrow links two spaces, name them $A$ and $B$; notice that the operator  $\Lambda_{(T_\iota, S_\iota)}$ is reduced by $A\oplus B$ and   the action of $\Lambda_{(T_\iota, S_\iota)}$ on $A\oplus B$ is, modulo a unitary equivalence, the same as $T_\iota$ acting of $H_\iota$ or
$S_\iota$ acting on $K_\iota$.

The following lemma summarizes the main properties of the map $ (T_\iota, S_\iota)\mapsto \Lambda_{(T_\iota, S_\iota)}$.

\begin{lemma}\label{Lemma:PropertiesLambdaOp}

Fix $\iota$ in $\{\alpha, \beta\}$ and let $P:H\to H$ denote the orthogonal projection onto $\mathbb{C}\xi\oplus H_\iota^o$.
 
The map
\begin{eqnarray*}
\mathbb{B}(H_\iota )\times \mathbb{B}(K_\iota) \to \mathbb{B}(H) \\
(T_\iota,S_\iota)\mapsto \Lambda_{(T_\iota,S_\iota),}
\end{eqnarray*}
satisfies:
\begin{enumerate}

\item For any operators $T_\iota, \tilde{T}_\iota, S_\iota, \tilde{S}_\iota$:
\begin{eqnarray*}
\Lambda_{(T_\iota, S_\iota)}P&=&\Lambda_{(T_\iota, \tilde{S}_\iota)}P ,\\
\Lambda_{(T_\iota,S_\iota)}P^\perp &=& \Lambda_{(\tilde{T}_\iota, S_\iota)}P^\perp.
\end{eqnarray*}

\item For any scalar $\alpha$ 
  $$\Lambda_{(\alpha T_\iota, \alpha S_\iota)}=\alpha\Lambda_{(T_\iota,S_\iota)}.$$

\item  For any operators $T_{\iota_1}, T_{\iota_2} \in \mathbb{B}(H_\iota), S_{\iota_1}, S_{\iota_2} \in \mathbb{B}(K_\iota)$
 \begin{eqnarray*} \Lambda_{(T_{\iota_1}+T_{\iota_2},S_{\iota_1})}P&=&\Lambda_{(T_{\iota_1},S_{\iota_1})}P+\Lambda_{(T_{\iota_2},S_{\iota_1})}P  ,\\
   \Lambda_{(T_{\iota_1},S_{\iota_1}+S_{\iota_2})}P^\perp &=& \Lambda_{( T_{\iota_1}     ,S_{\iota_1})}P^\perp +\Lambda_{(   T_{\iota_1}          ,S_{\iota_2})}P^\perp,\\
  \Lambda_{(T_{\iota_2}T_{\iota_1}, S_{\iota_2}S_{\iota_1})}&  =  & \Lambda_{(T_{\iota_2},S_{\iota_2})}\Lambda_{(T_{\iota_1},S_{\iota_1})}, \\
  \Lambda_{(T_\iota, S_\iota)}^*&=&\Lambda_{(T_\iota^*, S_\iota^*)}. 
 \end{eqnarray*}
    
  \item If $\Lambda_{(T,S)}=0$ then $S=0=T$.

\item $\| \Lambda_{(T_\iota, S_\iota)} \| \leq \max\{ \|T\| , \|S\|\}$.

\end{enumerate}

\end{lemma}
\begin{proof}
The proof of all properties follows from the next observation.

Fix $\iota$ in  $\{\alpha, \beta\}$. Given $\kappa$ in $\{\alpha, \beta\}$  and $\underline{l} $ in $I(\overline{\iota}; \kappa)$, let us denote

$$
H( \underline{l}) =K_{\underline{l}}^o\otimes H_\kappa^o \oplus K_\iota^o\otimes K_{\underline{l}}^o\otimes H_\kappa^o.
$$

 By construction, operators of the form $\Lambda_{(T_\iota,S_\iota)}$ are reduced by the sub-spaces $H(\underline{l})$.

\end{proof}

\begin{remark}
The maps introduced  in Definition \ref{Def:CFEmbeddings}, when specialized to each one of  the free, Boolean, monotone and orthogonal types of independence, correspond to the natural embeddings in  each  case. 
\end{remark}

 Lemma \ref{Lemma:CFAlternantingWordsAtxi} below is an important technical tool that we use to prove several results. In particular we use it to prove  Theorem \ref{Coro:PsiProducts} below, which gives an alternative proof to the proof of Theorem  \ref{Ex} (see \cite{BS}), in the case that $F$ contains only two elements.  Our result is  slightly  stronger than the one in Theorem \ref{Ex} because in our case we construct a state $ \hat \varphi  $ which is a vector state.  Lemma  \ref{Lemma:CFAlternantingWordsAtxi} is also useful to prove Theorem  \ref{Thm:FreeCopiesCFAlgebras}, which is one of our main results.

\begin{notation}
Let $\varphi_\xi, \varphi_{\xi_\alpha}$ and $\varphi_{\xi_\beta}$ denote the vector states at  the vectors $\xi \in H$, $\xi_\alpha \in H_\alpha, \xi_\beta\in H_\beta$, respectively.
\end{notation}

\begin{lemma}\label{Lemma:CFAlternantingWordsAtxi}

Let $r\geq 1$ be a natural number and let
\begin{eqnarray*}
\{ T_{\alpha(1)}, \dots, T_{\alpha(r)}\} \subset \mathbb{B}(H_\alpha), \quad \{ S_{\alpha(1)}, \dots, S_{\alpha(r)}\} \subset \mathbb{B}(K_\alpha), \\
 \{T_{\beta(1)}, \dots, T_{\beta(r)}\} \subset \mathbb{B}(H_\beta),\quad \{S_{\beta(1)},\dots,  S_{\beta(r)}\} \subset \mathbb{B}(K_\beta),
 \end{eqnarray*} 
 be subsets of operators such that: 
$$\langle S_{\alpha(j)}\eta_\alpha, \eta_\alpha \rangle=0=\langle S_{\beta(j)}\eta_\beta, \eta_\beta \rangle, \quad  \textrm{ for all $1 \leq j \leq r$}. $$

Take $\{ \Lambda_{i}\}_{1\leq i \leq n }$ an alternating sequence of elements from $\{ \Lambda_{(T_{\alpha(j)},S_{\alpha(j)})} \}_{j=1}^r$ and  $\{ \Lambda_{(T_{\beta(j)},S_{\beta(j)})} \}_{j=1}^r$. 

For simplicity we write $\Lambda_i=\Lambda_{(T_{\Lambda_i}, S_{\Lambda_i})}$ and
\begin{eqnarray*}
\xi_{\Lambda_i}=\left\{ 
\begin{array}{cc}
\xi_\alpha & \textrm{if $\Lambda_i$ is of the form $\Lambda_{(T_\alpha, S_\alpha)}$}, \\
\xi_\beta & \textrm{if $\Lambda_i$ is of the form $\Lambda_{(T_\beta, S_\beta)}$} .
\end{array}
\right.
\end{eqnarray*}

\begin{eqnarray*}
\begin{array}{cc}
\eta_{\Lambda_i}=\left\{ 
\begin{array}{cc}
\eta_\alpha & \textrm{if $\Lambda_i$ is of the form $\Lambda_{(T_\alpha, S_\alpha)}$}, \\
\eta_\beta & \textrm{if $\Lambda_i$ is of the form $\Lambda_{(T_\beta, S_\beta)}$}.
\end{array}
\right.
\end{array}
\end{eqnarray*}

With this notation we have:
\begin{eqnarray*}
\Lambda_n \cdots  \Lambda_1 \xi &=& (P_{\eta_{\Lambda_n}}^\perp S_{\Lambda_n} \eta_{\Lambda_n}) \otimes \cdots \otimes (P_{\eta_{\Lambda_2}}^\perp S_{\Lambda_2}\eta_{\Lambda_2})\otimes (P_{\xi_{\Lambda_1}}^\perp T_{\Lambda_1}\xi_{\Lambda_1}) \\
&+&(P_{\eta_{\Lambda_n}}^\perp S_{\Lambda_n} \eta_{\Lambda_n}) \otimes \cdots \otimes  (P_{\eta_{\Lambda_2}}^\perp T_{\Lambda_2}\xi_{\Lambda_2}) \varphi_{\xi_{\Lambda_1}}( T_{\Lambda_1}) \\
&\vdots & \\
&+& (P_{\xi_{\Lambda_n}}^\perp T_{\Lambda_n}\xi_{\Lambda_n}) \varphi_{\xi_{\Lambda_{n-1}}}(  T_{  \Lambda_{n-1}  }   ) \cdots \varphi_{\xi_{\Lambda_1}}(T_{\Lambda_1}) \\
&+& \xi  \varphi_{\xi_{\Lambda_n}}(T_{\Lambda_n})\varphi_{\xi_{\Lambda_{n-1}}}( 
 T_{  \Lambda_{n-1}  }  
 ) \cdots \varphi_{\xi_{\Lambda_1}}(T_{\Lambda_1}).
\end{eqnarray*}
Notice that, in order to make the result clearer, scalars are multiplied from the right.
\end{lemma}

\begin{proof}

We use induction on $n$.

For $n=1$, the definition of $\Lambda_1$ gives:
\begin{eqnarray*}
\Lambda_1 \xi &=& P_{\xi_{\Lambda_1}}^\perp T_{\Lambda_1}\xi_{\Lambda_1} + \langle T_{\Lambda_1}\xi_{\Lambda_1}, \xi_{\Lambda_1}\rangle \xi \\
&=& P_{\xi_{\Lambda_1}}^\perp T_{\Lambda_1}\xi_{\Lambda_1}  \\
&+&  \xi \varphi_\xi(\Lambda_1).
\end{eqnarray*}

We assume that the result  holds for $n$ and  prove it for $n+1$. By induction hypothesis we have:

\begin{eqnarray*}
\Lambda_{n+1}\Lambda_n \cdots \Lambda_1 \xi &=& \Lambda_{n+1}( (P_{\eta_{\Lambda_n}}^\perp S_{\Lambda_n} \eta_{\Lambda_n}) \otimes \cdots \otimes (P_{\eta_{\Lambda_2}}^\perp S_{\Lambda_2}\eta_{\Lambda_2})\otimes (P_{\xi_{\Lambda_1}}^\perp T_{\Lambda_1}\xi_{\Lambda_1}) ) \\
&+& \Lambda_{n+1}( (P_{\eta_{\Lambda_n}}^\perp S_{\Lambda_n} \eta_{\Lambda_n}) \otimes \cdots \otimes (P_{\eta_{\Lambda_2}}^\perp T_{\Lambda_2}\xi_{\Lambda_2}) ) \varphi_{\xi_{\Lambda_1}}( T_{\Lambda_1}) \\
&\vdots & \\
&+& \Lambda_{n+1}(P_{\xi_{\Lambda_n}}^\perp T_{\Lambda_n}\xi_{\Lambda_n}) \varphi_{\xi_{\Lambda_{n-1}}}
(  T_{\Lambda_{n-1}}     )
 \cdots \varphi_{\xi_{\Lambda_1}}(T_{\Lambda_1}) \\
&+& \Lambda_{n+1}(\xi) \varphi_{\xi_{\Lambda_n}}
(T_{\Lambda_n})\varphi_{\xi_{\Lambda_{n-1}}}
( T_{\Lambda_{n-1}}    )
 \cdots \varphi_{\xi_{\Lambda_1}}(T_{\Lambda_1}). 
\end{eqnarray*}

Fix $j$ with $1\leq j \leq n$ and let us denote
$$
w=(P_{\eta_{\Lambda_n}}^\perp S_{\Lambda_n} \eta_{\Lambda_n}) \otimes \cdots \otimes 
(P_{\eta_{\Lambda_{j+1}}}^\perp S_{\Lambda_{j+1}}\eta_{\Lambda_{j+1}})     
\otimes (P_{\xi_{\Lambda_j}}^\perp T_{\Lambda_j}\xi_{\Lambda_j}).
$$

Since the $\Lambda_i$'s are alternating, it follows from the definition of the  operators $S_{\Lambda}$'s that 
\begin{eqnarray*}
\Lambda_{n+1}(w)&=&\langle S_{\Lambda_{n+1}} \eta_{\Lambda_{n+1}}, \eta_{\Lambda_{n+1}} \rangle w+ (P_{\eta_{\Lambda_{n+1}}}^\perp S_{\Lambda_{n+1}}\eta_{\Lambda_{n+1}})\otimes w \\
&=& 0+ (P_{\eta_{\Lambda_{n+1}}}^\perp S_{\Lambda_{n+1}}\eta_{\Lambda_{n+1}})\otimes w.
\end{eqnarray*}

Hence we have

\begin{eqnarray*}
\Lambda_{n+1}\Lambda_n \cdots \Lambda_1 \xi &=&  (P_{\eta_{\Lambda_{n+1}}}^\perp S_{\Lambda_{n+1}}\eta_{\Lambda_{n+1}}) \otimes (P_{\eta_{\Lambda_n}}^\perp S_{\Lambda_n} \eta_{\Lambda_n}) \otimes \cdots \\ 
& \cdots &  \otimes(P_{\eta_{\Lambda_2}}^\perp S_{\Lambda_2}\eta_{\Lambda_2})\otimes (P_{\xi_{\Lambda_1}}^\perp T_{\Lambda_1}\xi_{\Lambda_1})  \\
&+& (P_{\eta_{\Lambda_{n+1}}}^\perp S_{\Lambda_{n+1}}\eta_{\Lambda_{n+1}}) \otimes  (P_{\eta_{\Lambda_n}}^\perp S_{\Lambda_n} \eta_{\Lambda_n}) \otimes \cdots  \\
& \cdots &  \otimes (P_{\eta_{\Lambda_3}}^\perp S_{\Lambda_3}\eta_{\Lambda_3})\otimes (P_{\eta_{\Lambda_2}}^\perp T_{\Lambda_2}\xi_{\Lambda_2})  \varphi_{\xi_{\Lambda_1}}( T_{\Lambda_1}) \\
&\vdots & \\
&+& (P_{\eta_{\Lambda_{n+1}}}^\perp S_{\Lambda_{n+1}}\eta_{\Lambda_{n+1}}) \otimes (P_{\xi_{\Lambda_n}}^\perp T_{\Lambda_n}\xi_{\Lambda_n}) \varphi_{\xi_{\Lambda_{n-1}}}
(  T_{ \Lambda_{n-1} }    )
 \cdots \\ 
&\cdots & \varphi_{\xi_{\Lambda_1}}(T_{\Lambda_1}) \\
&+& 
\Lambda_{n+1}(\xi) ( \varphi_{\xi_{\Lambda_n}}(T_{\Lambda_n})\varphi_{\xi_{\Lambda_{n-1}}}
(T_{\Lambda_{n-1}}    )
 \cdots \varphi_{\xi_{\Lambda_1}}(T_{\Lambda_1}) )
\end{eqnarray*}
Lastly, using
$$
\Lambda_{n+1}(\xi)= (P_{\xi_{\Lambda_{n+1}}}^\perp T_{\Lambda_{n+1}}\xi_{\Lambda_{n+1}})+ \xi \varphi_{\xi_{\Lambda_{n+1}}}(T_{\Lambda_{n+1}}),
$$
we conclude that

\begin{eqnarray*}
\Lambda_{n+1}\Lambda_n \cdots \Lambda_1 \xi &=&  (P_{\eta_{\Lambda_{n+1}}}^\perp S_{\Lambda_{n+1}}\eta_{\Lambda_{n+1}}) \otimes (P_{\eta_{\Lambda_n}}^\perp S_{\Lambda_n} \eta_{\Lambda_n}) \otimes \cdots \\ 
&\cdots & \otimes (P_{\eta_{\Lambda_2}}^\perp S_{\Lambda_2}\eta_{\Lambda_2})\otimes (P_{\xi_{\Lambda_1}}^\perp T_{\Lambda_1}\xi_{\Lambda_1})  \\
&+& (P_{\eta_{\Lambda_{n+1}}}^\perp S_{\Lambda_{n+1}}\eta_{\Lambda_{n+1}}) \otimes  (P_{\eta_{\Lambda_n}}^\perp S_{\Lambda_n} \eta_{\Lambda_n}) \otimes \cdots  \\
&\cdots & \otimes (P_{\eta_{\Lambda_2}}^\perp T_{\Lambda_2}\xi_{\Lambda_2})  \varphi_{\xi_{\Lambda_1}}( T_{\Lambda_1}) \\
&\vdots & \\
&+& (P_{\eta_{\Lambda_{n+1}}}^\perp S_{\Lambda_{n+1}}\eta_{\Lambda_{n+1}}) \otimes (P_{\xi_{\Lambda_n}}^\perp T_{\Lambda_n}\xi_{\Lambda_n}) \varphi_{\xi_{\Lambda_{n-1}}}(\Lambda_{n-1}) \cdots \\ 
&\cdots &  \varphi_{\xi_{\Lambda_1}}(T_{\Lambda_1}) \\
&+&  (P_{\xi_{\Lambda_{n+1}}}^\perp T_{\Lambda_{n+1}}\xi_{\Lambda_{n+1}}) \varphi_{\xi_{\Lambda_n}}(T_{\Lambda_n})\varphi_{\xi_{\Lambda_{n-1}}}  
( T_{   {\Lambda_{n-1}}  }   )
 \cdots  \\
&\cdots & \varphi_{\xi_{\Lambda_1}}(T_{\Lambda_1})  \\ 
&+& \xi \varphi_{\xi_{\Lambda_{n+1}}}(T_{\Lambda_{n+1}}) \varphi_{\xi_{\Lambda_n}}(T_{\Lambda_n})\varphi_{\xi_{\Lambda_{n-1}}}
( T_{\Lambda_{n-1} }  )
 \cdots \varphi_{\xi_{\Lambda_1}}(T_{\Lambda_1}). 
\end{eqnarray*}

\end{proof}

Next theorem (Theorem   \ref{Coro:PsiProducts}) is a sligthly stronger version of  Theorem \ref{Ex} (see \cite{BS}) because the state $ \hat \varphi $ that we construct is a vector state. It is the main result in this section. The proof Theorem   \ref{Coro:PsiProducts}  uses  technology that we develop in this paper and it is, therefore, different than the one in \cite{BS}. In particular,    Lemma  \ref{Lemma:CFAlternantingWordsAtxi} is the key ingredient in it.   

\begin{thm}\label{Coro:PsiProducts}
Let $(\mathcal{A}_\alpha, \varphi_\alpha, \psi_\alpha)$ and $(\mathcal{A}_\beta, \varphi_\beta, \psi_\beta)$ denote unital $C^*$-algebras with   pairs of states  $ (\varphi_\alpha, \psi_\alpha)  $ and $ (\varphi_\beta, \psi_\beta)  $, respectively. For each one of the states above, we denote by 

$(H_\iota, \xi_\iota)$ the GNS representation of $\varphi_\iota$, $\pi_\iota : \mathcal{A}_\iota \to \mathbb{B}(H_\iota)$,

$(K_\iota, \eta_\iota)$ the GNS representation of $\psi_\iota$, $\sigma_\iota : \mathcal{A}_\iota \to \mathbb{B}(K_\iota).$

We set
$$((H,\xi),(K,\eta)):=((H_\alpha,\xi_\alpha),(K_\alpha,\eta_\alpha))*((H_\beta,\xi_\beta),(K_\beta,\eta_\beta))$$
and let  $\rho_\iota:\mathcal{A}_\iota \to \mathbb{B}(H)$ be given by  $\rho_\iota(x)=\Lambda_{(\pi_\iota(x),\sigma_\iota(x))}$. We define $\hat{\mathcal{A}}$ as the unital $C^*$-algebra generated by $\rho_\iota(\mathcal{A}_\iota)$, $\iota=\alpha, \beta$ and let $\hat{\varphi}$ denote the vector state of $\mathbb{B}(H)$ induced by $\xi$.
 
 Then, $(\hat{\mathcal{A}}, \hat{\varphi})$ is a reduced $\psi$-product for $(\mathcal{A}_\alpha, \varphi_\alpha, \psi_\alpha)$ and $(\mathcal{A}_\beta, \varphi_\beta, \psi_\beta)$.
\end{thm}

\begin{proof}
From the definition of the operator $\rho_\iota(x)$ it is straightforward to verify  that $\hat{\varphi}(\rho_\iota(x))=\varphi_\iota(x)$. Hence, from Lemma \ref{Lemma:CFAlternantingWordsAtxi} we deduce the factorization property for $\hat{\varphi}$ (Eq. \eqref{fact}).

From Lemma \ref{Lemma:PropertiesLambdaOp}, it follows that  $\rho_\iota$ is a unital and injective  $*$-representation of $\mathcal{A}$. 
 The only property that needs to be checked is addition. We prove it for $\iota = \alpha$, the case $  \iota = \beta $   is analogous. As in Lemma \ref{Lemma:PropertiesLambdaOp}, let $P$ denote the orthogonal projection onto $\mathbb{C}\xi\oplus H_\alpha^o$, then using Lemma \ref{Lemma:PropertiesLambdaOp} we have:

\begin{eqnarray*}
\rho_\alpha(x+y)&=& \Lambda_{(\tau_\alpha(x+y), \sigma_\alpha(x+y))} \\
&=& \Lambda_{(\tau_\alpha(x+y), \sigma_\alpha(x+y))}P + \Lambda_{(\tau_\alpha(x+y), \sigma_\alpha(x+y))}P^\perp \\
&=& \Lambda_{(\tau_\alpha(x),\sigma_\alpha(x+y) )}P + \Lambda_{(\tau_\alpha(y), \sigma_\alpha(x+y))}P \\
&+& \Lambda_{(\tau_\alpha(x+y), \sigma_\alpha(x))}P^\perp + \Lambda_{(\tau_\alpha(x+y), \sigma_\alpha(y))}P^\perp \\
&=& \Lambda_{(\tau_\alpha(x),\sigma_\alpha(x) )}P + \Lambda_{(\tau_\alpha(y), \sigma_\alpha(y))}P \\
&+& \Lambda_{(\tau_\alpha(x), \sigma_\alpha(x))}P^\perp + \Lambda_{(\tau_\alpha(y), \sigma_\alpha(y))}P^\perp \\
&=& \Lambda_{(\tau_\alpha(x),\sigma_\alpha(x))}+ \Lambda_{(\tau_\alpha(y), \sigma_\alpha(y))} \\
&=&\rho_\alpha(x)+\rho_\alpha(y).
\end{eqnarray*}

\end{proof}

\section{Free copies of c. free algebras}\label{Section4}

\subsection{Free copies}

 In Section \ref{Section3} and specially in Theorem \ref{Coro:PsiProducts} we prove that the reduced $\psi$-product $(\hat{\mathcal{A}}, \hat \varphi )$ of Theorem \ref{Ex} can be chosen in such a way that $ \hat{\varphi} $ is a vector state. Here we address a similar question, but instead of Theorem \ref{Ex} we consider Definition \ref{Def1} (c.  free algebras). More precisely, we prove that the vectors  $ \varphi, \psi$ in Definition \ref{Def1} can be selected to be vector states, in the case that $F$ contains only two elements. Here we need to replace the algebras $\mathcal{A}_{\iota}$ by copies of them and the states $ \varphi $ and $\psi$ are substituted by vector states $ \varphi_{\xi}  $ and $\psi_{\tilde \eta}$, respectively ($ \varphi_{\xi} $ is the vector state at $\xi$ and  $\psi_{\tilde \eta}$ is the vector state at $\tilde \eta $).   Moreover, we prove that these copies of the algebras  $\mathcal{A}_{\iota} $   are free with respect to $  \psi_{\tilde \eta}  $.

\begin{lemma}\label{Lemma:AgreementWRTxi}
For any pair of operators $T_\alpha,S_\alpha$,
\begin{eqnarray*}
\langle \Lambda_{(T_\alpha,S_\alpha)} \xi, \xi \rangle= \langle T_\alpha \xi_\alpha, \xi_\alpha \rangle.
\end{eqnarray*}
A similar result holds true if we replace $\alpha $ by $\beta$.
\end{lemma}
\begin{proof}
We prove the result for $\alpha$, the corresponding result for $\beta$ is proved analogously.  

\begin{eqnarray*}
\langle \Lambda_{(T_\alpha,S_\alpha)} \xi, \xi \rangle &=& \langle \langle T_\alpha \xi_\alpha, \xi_\alpha \rangle \xi + P_{\xi_\alpha}^\perp T_\alpha \xi_\alpha, \xi \rangle \\
&=& \langle T_\alpha \xi_\alpha, \xi_\alpha \rangle.
\end{eqnarray*}
\end{proof}

\begin{lemma}\label{Lemma:AgreementWRTeta}
For any unit vectors, $k_i^o \in K_i^o , h_i^o \in H_i^o$, $i=\alpha, \beta$ and operators $T_\alpha,S_\alpha, T_\beta, S_\beta$:

\begin{eqnarray*}
\langle \Lambda_{(T_\alpha,S_\alpha)} k_\beta^o\otimes h_\alpha^o , k_\beta^o\otimes h_\alpha^o \rangle &= &\langle S_\alpha \eta_\alpha, \eta_\alpha \rangle, \\
\langle \Lambda_{(T_\beta,S_\beta)} k_\alpha^o\otimes h_\beta^o , k_\alpha^o\otimes h_\beta^o \rangle &= &\langle S_\beta \eta_\beta, \eta_\beta \rangle.
\end{eqnarray*}

\end{lemma}

\begin{proof}

\begin{eqnarray*}
\Lambda_{(T_\alpha,S_\alpha)}( k_\beta^o\otimes h_\alpha^o)&=& \langle S_\alpha \eta_\alpha, \eta_\alpha \rangle k_\beta^o\otimes h_\alpha^o  \\
&+&  (P_{\eta_\alpha}^\perp S_\alpha \eta_\alpha) \otimes k_\beta^o \otimes h_\alpha^o, \\
\Lambda_{(T_\beta,S_\beta)}( k_\alpha^o\otimes h_\beta^o)&=&  \langle S_\beta \eta_\beta, \eta_\beta \rangle k_\alpha^o\otimes h_\beta^o \\
&+&  (P_{\eta_\beta}^\perp S_\beta \eta_\beta )\otimes    k_\alpha^o  \otimes h_\beta^o.
\end{eqnarray*}
\end{proof}

\begin{lemma}\label{Lemma:AgreementWRTtildeeta}
For any unit vectors, $k_i^o \in K_i^o , h_i^o \in H_i^o$, $i=\alpha, \beta$ and operators $T_\alpha,S_\alpha, T_\beta, S_\beta$:

\begin{eqnarray*}
\langle \Lambda_{(T_\beta,S_\beta)} (k_\beta^o\otimes h_\alpha^o) , k_\beta^o\otimes h_\alpha^o \rangle &= &\langle S_\beta k_\beta^o, k_\beta^o \rangle,\\
\langle \Lambda_{(T_\alpha,S_\alpha)}( k_\alpha^o\otimes h_\beta^o) , k_\alpha^o\otimes h_\beta^o \rangle &= &\langle S_\alpha  k_\alpha^o, k_\alpha^o \rangle.
\end{eqnarray*}

\end{lemma}

\begin{proof}

We only prove the first identity, since the proof of the second is  similar.

A direct computation  brings

\begin{eqnarray*}
\Lambda_{(T_\beta,S_\beta)}( k_\beta^o\otimes h_\alpha^o)=  \langle S_\beta k_\beta^o, \eta_\beta \rangle h_\alpha^o + (P_{\eta_\beta}^\perp S_\beta k_\beta^o)\otimes h_\alpha^o,
\end{eqnarray*}
hence
\begin{eqnarray*}
\langle \Lambda_{(T_\beta,S_\beta)} (k_\beta^o\otimes h_\alpha^o) , k_\beta^o\otimes h_\alpha^o \rangle &=&  \langle (P_{\eta_\beta}^\perp S_\beta k_\beta^o)\otimes h_\alpha^o , k_\beta^o\otimes h_\alpha^o \rangle \\
&=& \langle P_{\eta_\beta}^\perp S_\beta k_\beta^o , k_\beta^o \rangle \\
&=& \langle S_\beta k_\beta^o, P_{\eta_\beta}^\perp k_\beta ^o \rangle \\
&=& \langle S_\beta k_\beta^o , k_\beta^o \rangle.
\end{eqnarray*}

\end{proof}

We are now  ready to prove  the main theorem in this section.

\begin{thm}\label{Thm:FreeCopiesCFAlgebras}
Let $(\mathcal{A},\varphi, \psi)$ be a unital $*$-algebra with two states and assume that:
\begin{enumerate}
\item $\mathcal{A}=\textrm{span}\{ u\in \mathcal{A}: \textrm{$u$ is a unitary of $\mathcal{A}$  } \}$;

\item $\varphi$ is not a unital $*$-homomorphism.

\end{enumerate}

Given $\mathcal{A}_\alpha , \mathcal{A}_\beta \subset  \mathcal{A}$,  unital  $*$-sub-algebras, conditionally free w.r.t. $(\varphi, \psi)$, there are: a Hilbert space $H$; unit vectors $\xi, \tilde{\eta}$  in $H$ and unital, injective,  $*$-representations $\rho_{\alpha}: \mathcal{A} \to \mathbb{B}(H)$, $\rho_\beta : \mathcal{A} \to \mathbb{B}(H)$,  such that:
\begin{enumerate}
\item For all $a$ in $\mathcal{A}_\alpha$ and $b$  $\mathcal{A}_\beta$,
\begin{eqnarray*}
\varphi(a)= \langle \rho_\alpha(a) \xi, \xi \rangle , \quad \varphi(b)=\langle \rho_\beta(b)\xi, \xi \rangle, \\
\psi(a)= \langle \rho_\alpha(a) \tilde{\eta}, \tilde{\eta} \rangle , \quad \psi(b)=\langle \rho_\beta(b) \tilde{\eta}, \tilde{\eta} \rangle .
\end{eqnarray*}
\item $\rho_\alpha(\mathcal{A}_\alpha)$ and $\rho_\beta(\mathcal{A}_\beta)$ are c. free w.r.t. the vector states at ($\xi$, $\tilde{\eta}$).
\item $\rho_\alpha(\mathcal{A}_\alpha)$ and $\rho_\beta(\mathcal{A}_\beta)$ are free w.r.t the vector state at $\tilde{\eta}$.
\end{enumerate}

\end{thm}

\begin{proof}

The assumption
$$
\mathcal{A}=\textrm{span}\{ u\in \mathcal{A}: \textrm{$u$ is a unitary of $\mathcal{A}$  } \}
$$
and the GNS representation imply  that there are    pointed Hilbert spaces $(H_\alpha, \xi_\alpha),(K_\alpha, \eta_\alpha) $ and unital $*$-representations $\tau_\alpha: \mathcal{A} \to \mathbb{B}(H_\alpha), \sigma_\alpha:\mathcal{A} \to \mathbb{B}(K_\alpha)$, such that, for all $x\in \mathcal{A}$,
\begin{eqnarray}
\varphi(x)=\langle \tau_\alpha (x)\xi_\alpha,\xi_\alpha\rangle, \psi(x)=\langle \sigma_\alpha(x) \eta_\alpha, \eta_\alpha \rangle \label{Eqn:ExtentionAlphaState}
\end{eqnarray}
(see Proposition 7.2 in \cite{NS}).

Next, we take $H_\beta:=H_\alpha\oplus H_\alpha$, $K_\beta:=K_\alpha\oplus K_\alpha$, $\xi_\beta:=(\xi_\alpha,0), \eta_\beta:=(\eta_\alpha,0)$. Consider the representations $\tau_\beta: \mathcal{A}\to \mathbb{B}(H_\beta)$ and
$\sigma_\beta: \mathcal{A}\to \mathbb{B}(K_\beta)$ given by $\tau_\beta(x):=\tau_\alpha(x)\oplus \tau_\alpha(x)$ and $\sigma_\beta(x):=\sigma_\alpha(x)\oplus \sigma_\alpha(x)$, respectively. Then, we also have that for all $x\in \mathcal{A}$,
\begin{eqnarray}
\varphi(x)=\langle \tau_\beta (x)\xi_\beta,\xi_\beta\rangle, \psi(x)=\langle \sigma_\beta(x) \eta_\beta, \eta_\beta \rangle = \langle \sigma_\beta(x)\eta_\beta^\perp, \eta_\beta^\perp \rangle \label{Eqn:ExtentionBetaState},
\end{eqnarray}
where we denote by $\eta_\beta^\perp :=(0,\eta_\alpha)$ and notice that $\eta_\beta^\perp $ is a unit vector in $K_\beta^o$.

Now, we construct the c. free product of pointed Hilbert spaces:
$$
((H,\xi), (K,\eta)):=((H_\alpha,\xi_\alpha), (K_\alpha, \eta_\alpha))*((H_\beta,\xi_\beta),(K_\beta, \eta_\beta)).
$$

Since $\varphi$ is not a unital homomorphism, $\dim(H_\alpha)>1$. Let  $h_\alpha^o \in H_\alpha^o$ be  a unit vector. Now we set $\tilde{\eta}:=\eta_\beta^\perp \otimes h_\alpha^o$.

We define the maps   $\rho_\iota : \mathcal{A}  \to  \mathbb{B}(H) $ given by $\rho_\iota(x):=\Lambda_{(\tau_\iota(x), \sigma_\iota(x))} $, $\iota=\alpha,\beta$. In a similar way as in the proof of Theorem \ref{Coro:PsiProducts}, it can be  showed that these maps are  unital, injective, $*$-representations of $\mathcal{A}$.
Hence, as unital $*$-algebras, $\mathcal{A}_\alpha$ is isomorphic to  $\rho_\alpha(\mathcal{A}_\alpha)$ and  $\mathcal{A}_\beta$ is isomorphic to $\rho_\beta(\mathcal{A}_\beta)$.

From Lemma \ref{Lemma:AgreementWRTxi}, \eqref{Eqn:ExtentionAlphaState}  and \eqref{Eqn:ExtentionBetaState}  it follows that for any $a$ in $\mathcal{A}_\alpha$ and $b$ in $\mathcal{A}_\beta$

$$
 \langle \rho_\alpha(a)  \xi, \xi \rangle =\varphi(a),   \quad \langle \rho_\beta(b ) \xi, \xi \rangle=\varphi(b).
$$

Lemma \ref{Lemma:AgreementWRTeta} along with \eqref{Eqn:ExtentionAlphaState} imply that, for all $a$ in $\mathcal{A}_\alpha$,
$$
\langle \rho_\alpha(a) \tilde{\eta}, \tilde{\eta} \rangle = \psi(a).
$$

Now, for all $b$ in $\mathcal{A}_\beta$, we prove that
$$\langle \rho_\beta(b) \tilde{\eta}, \tilde{\eta} \rangle=\psi(b). $$
This is a consequence of  Lemma \ref{Lemma:AgreementWRTtildeeta} and \eqref{Eqn:ExtentionBetaState}, because
\begin{eqnarray*}
\langle \rho_\beta(b) \tilde{\eta}, \tilde{\eta} \rangle&=& \langle \Lambda_{(\tau_\beta(b),\sigma_\beta(b))} \tilde{\eta}, \tilde{\eta} \rangle \\
&=&  \langle (\Lambda_{(\tau_\beta(b), \sigma_\beta(b))}(\eta_\beta^\perp \otimes h_\alpha^o) , \eta^\perp_\beta\otimes h_\alpha^o \rangle \\
&=& \langle \sigma_\beta(b)\eta_\beta^\perp , \eta_\beta^\perp\rangle \\
&=& \psi(b).
\end{eqnarray*}

From Lemma \ref{Lemma:CFAlternantingWordsAtxi}, it follows that $\rho_\alpha(\mathcal{A}_\alpha)$ and $\rho_\beta(\mathcal{A}_\beta)$
are c. free w.r.t. the vector states at $(\xi, \tilde{\eta})$.

What remains to be proven is that $\rho_\alpha(\mathcal{A}_\alpha)$ and $\rho_\beta(\mathcal{A}_\beta)$ are free w.r.t. the vector state at $\tilde{\eta}$, i.e. they satisfy Definition \ref{Voicu}. It is, therefore, necessary to consider alternating products of elements of   $\rho_\alpha(\mathcal{A}_\alpha)$ and $\rho_\beta(\mathcal{A}_\beta)$. In this proof, we only analyze products that start with elements of  $\rho_\alpha(\mathcal{A}_\alpha)$ and end with elements of $\rho_\beta(\mathcal{A}_\beta)$. This, of course, does not cover all possible cases, but the other cases can be studied similarly. The desired result is a direct consequence of Eq.  \eqref{Eqn:ProofFreeIndependeceAttildeeta} below.  Now we introduce some notation in order to simplify our formulas.

We denote by $\Lambda_a$ the elements  of $\rho_\alpha(\mathcal{A}_\alpha)$ and by $\Lambda_b$ the elements of 
$\rho_\beta(\mathcal{A}_\beta)$. Hence, we can write
$$
\Lambda_a=\Lambda_{(T_a, S_a)}, \quad \Lambda_b=\Lambda_{(T_b, S_b)},
$$
where $T_a=\tau_\alpha(a), S_a=\sigma_\alpha(a), T_b=\tau_\beta(b), S_b=\sigma_\beta(b)$.

 We show that 
\begin{eqnarray}
\Lambda_{b(r)}\Lambda_{a(r)}\cdots \Lambda_{b(1)}\Lambda_{a(1)}\tilde{\eta} &=& (P_{\eta_\beta}^\perp S_{b(r)}\eta_\beta)\otimes (P_{\eta_\alpha}^\perp S_{a(r)}\eta_\alpha)\otimes   \cdots  \nonumber \\
& \cdots & \otimes (P_{\eta_\beta}^\perp S_{b(1)}\eta_\beta) \otimes (P_{\eta_\alpha}^\perp S_{a(1)}\eta_\alpha)\otimes \tilde{\eta},
\label{Eqn:ProofFreeIndependeceAttildeeta}
\end{eqnarray}
whenever $r\geq 1$, $a(1), \dots , a(r)$ are elements in $\mathcal{A}_{\alpha}$ and $b(1),\dots, b(r)$ are elements in $\mathcal{A}_\beta$ that satisfy
$$
\langle \Lambda_{a(j)} \tilde{\eta}, \tilde{\eta}\rangle =0=\langle \Lambda_{b(j)}\tilde{\eta}, \tilde{\eta}\rangle , \quad \textrm{for all $j=1,\dots, r$.}
$$

To prove \eqref{Eqn:ProofFreeIndependeceAttildeeta} we proceed by induction in $r$.

For $r=1$, we have:

\begin{eqnarray*}
\Lambda_{(T_{a(1)}, S_{a(1)})}\tilde{\eta}&=&\langle S_{a(1)}\eta_\alpha, \eta_\alpha \rangle \tilde{\eta}+ (P_{\eta_\alpha}^\perp S_{a(1)}\eta_\alpha)\otimes \tilde{\eta} \\
&=& 
\langle \Lambda_{a(1)}\tilde{\eta}, \tilde{\eta}\rangle            
    \tilde{\eta}  
 + (P_{\eta_\alpha}^\perp S_{a(1)}\eta_\alpha)\otimes \tilde{\eta}\\
&=& (P_{\eta_\alpha}^\perp S_{a(1)}\eta_\alpha)\otimes \tilde{\eta},
\end{eqnarray*}

which implies

\begin{eqnarray*}
\Lambda_{(T_{b(1)}, S_{b(1)})}   
\Lambda_{(T_{a(1)}, S_{a(1)})} 
\tilde{\eta}&=& \Lambda_{(T_{b(1)}, S_{b(1)})}((P_{\eta_\alpha}^\perp S_{a(1)}\eta_1)\otimes \tilde{\eta}) \\
&=&  \langle S_{b(1)} \eta_\beta, \eta_\beta \rangle (P_{\eta_\alpha}^\perp S_{a(1)}\eta_\alpha)\otimes \tilde{\eta}
\\
&+& (P_{\eta_\beta}^\perp S_{b(1)}\eta_\beta)\otimes (P_{\eta_\alpha}^\perp S_{a(1)}\eta_\alpha)\otimes \tilde{\eta}\\
&=& \langle \Lambda_{b(1)} \tilde{\eta}, \tilde{\eta} \rangle (P_{\eta_\alpha}^\perp S_{a(1)}\eta_\alpha)\otimes \tilde{\eta}\\
&+& (P_{\eta_\beta}^\perp S_{b(1)}\eta_\beta)\otimes (P_{\eta_\alpha}^\perp S_{a(1)}\eta_\alpha)\otimes \tilde{\eta} \\
&=& (P_{\eta_\beta}^\perp S_{b(1)}\eta_\beta)\otimes (P_{\eta_\alpha}^\perp S_{a(1)}\eta_\alpha)\otimes \tilde{\eta}.
\end{eqnarray*}

Now, assume that \eqref{Eqn:ProofFreeIndependeceAttildeeta}  holds and let
$$
X:=(P_{\eta_\beta}^\perp S_{b(r)}\eta_\beta)\otimes (P_{\eta_\alpha}^\perp S_{a(r)}\eta_\alpha)\otimes   \cdots  \otimes (P_{\eta_\beta}^\perp S_{b(1)}\eta_\beta) \otimes (P_{\eta_\alpha}^\perp S_{a(1)}\eta_\alpha)\otimes \tilde{\eta}.
$$

Then
\begin{eqnarray*}  
\Lambda_{(T_{a(r+1)}, S_{a(r+1)} )}    
(X)
&=& \langle S_{a(r+1)}\eta_\alpha, \eta_\alpha \rangle X +(P_{\eta_\alpha}^\perp S_{a(r+1)}\eta_\alpha)\otimes X \\
&=&(P_{\eta_\alpha}^\perp S_{a(r+1)}\eta_\alpha)\otimes X.
\end{eqnarray*}
and
\begin{eqnarray*}
\Lambda_{(T_{b(r+1)}, S_{b(r+1)})}\Lambda_{(T_{a(r+1)}, S_{a(r+1)})}(X)&=&\langle S_{b(r+1)} \eta_\beta, \eta_\beta \rangle(P_{\eta_\alpha}^\perp S_{a(r+1)}\eta_\alpha)\otimes X   \\
&+& (P_{\eta_\beta}^\perp S_{b(r+1)}\eta_\beta)\otimes (P_{\eta_\alpha}^\perp S_{a(r+1)}\eta_\alpha)\otimes X  \\
&=&(P_{\eta_\beta}^\perp S_{b(r+1)}\eta_\beta)\otimes (P_{\eta_\alpha}^\perp S_{a(r+1)}\eta_\alpha)\otimes X . 
\end{eqnarray*}

\end{proof}

\section{The $^cR$-transform and its linearization property}\label{SEction5}

As an application of Theorem \ref{Thm:FreeCopiesCFAlgebras}  we give an alternative proof of the linearization property of the conditionally free $R$-transform. We follow the proof of Haagerup  in \cite{Haa} and in order to avoid repetitions we only present the arguments that use  tools  introduced in this paper.   In  this section,  we denote by $(\mathcal{A}, \varphi, \psi)$  a unital  $C^*$-algebra with two states.

\begin{notation} 
For an element  $a$ in $\mathcal{A}$ and  a complex number  $t$ with $|t|< \frac{1}{\|a\|}$
, denote
\begin{eqnarray*}
h_a(t)&:=& \psi((1-ta)^{-1}), \quad \tilde{h}_a(t):= \varphi((1-ta)^{-1}),\\
k_a(t)&:=&th_a(t),\quad \tilde{k}_a(t):=t\tilde{h}_a(t).
\end{eqnarray*}
\end{notation}

The following lemma is essentially  Proposition 3.1-(a)  in \cite{Haa} and the only  important observation to make is that the bijectivity of the functions $k_a, \tilde{k}_a$ does not rely on the state that we use to define them.

\begin{lemma}
$k_a$ and $\tilde{k}_a$  are bijections from  $B_{\frac{1}{4\|a\|}}(0)$ onto a neighbourhood of 0 which contains $B_{\frac{1}{6\|a\|}}(0)$.  
Here we use the notation $  B_r(z) := \{  t \in \mathbb{C}  :   |t - z| < r \} .$
\end{lemma}

Using this lemma we can give a definition for the conditionally free $R$-transform, with respect to $(\varphi,\psi)$.

\begin{definition}\label{Def:cRCstarAlg}
For $a$ in $\mathcal{A}$,  its c. free $R$-transform, denoted by $R^{(\varphi,\psi)}_a$ or $^cR_a$, is defined by 
\begin{eqnarray}\label{Eqn:cR-T}
R_{a}^{(\varphi,\psi)}(z)=\left\{ \begin{array}{cc}
\frac{1}{k_a^{-1}(z)}-\frac{1}{\tilde{k}_a(k_a^{-1}(z))} &  0<|z|< \frac{1}{6\|a\|}, \\
\empty \\
\varphi(a) & z=0.
\end{array}
\right.
\end{eqnarray}

We denote by
$$
G_a^{\psi}(\lambda ):=\sum_{n=0}^\infty \lambda^{-n-1}\psi(a^n)=\psi((\lambda 1_\mathcal{A}-a)^{-1}),   \hspace{2cm} |\lambda| >\|a\| ,
$$
the Cauchy transform of $a$ with respect to $\psi$, which is  an analytic function.
 It is straightforward to  verify that $ k_a(t) = G_{a}^{\psi}(1/t) $ and, therefore,    $(G_{a}^\psi)^{-1}(z)=\frac{1}{k_a^{-1}(z)}$. Hence we can rewrite \eqref{Eqn:cR-T} as
\begin{align}\label{Exta}
R_a^{(\varphi,\psi)}(G_a^\psi(\lambda))=\lambda - \frac{1}{G_a^\varphi(\lambda)},
\end{align}
where we use that  $  \tilde k_a(t) = G_{a}^{\varphi}(1/t)    $. Eq. \eqref{Exta}  is the analytic description given in Theorem 5.2 of \cite{BLS}.
\end{definition}

\begin{definition}
Take $a, b \in \mathcal{A}$   and let $\mathcal{A}_a$ and $\mathcal{A}_b$ denote the unital $C^*$-algebras generated by $a$ and $b$ respectively. We say that $a$ and $b$ are c. free with respecto to $(\varphi, \psi)$ if $\mathcal{A}_a$ and $\mathcal{A}_b$ are c. free with respecto to $(\varphi, \psi)$.
\end{definition}

\begin{lemma}
Let $a$, $b$ be two  elements c. free with respect to   $(\varphi,\psi)$.  Denote
\begin{eqnarray}\label{ab}
a(t_1)&=&(1-t_1a)^{-1}-h_a(t_1)1_\mathcal{A}, \quad |t_1|< \frac{1}{\|a\|}, \notag \\
b(t_2)&=&(1-t_2b)^{-1}-h_b(t_2)1_\mathcal{A}, \quad |t_2|< \frac{1}{\|b\|}.  
\end{eqnarray}

Then
\begin{enumerate}
\item $a(t_1), b(t_2)$ are c. free with respect to $(\varphi,\psi)$  and  $\psi(a(t_1))=\psi(b(t_2))=0$.
\item For  every  $\rho \in \mathbb{C}$:
$$
(1-t_1a)(1-\rho a(t_1)b(t_2))(1-t_2b)=c_0+c_1a+c_2b+c_3ab
$$
where:
\begin{eqnarray*}
c_0=c_0(\rho, t_1,t_2)&=&1-\rho(h_a(t_1)-1)(h_b(t_2)-1), \\
c_1=c_1(\rho, t_1,t_2)&=&-t_1(1+\rho h_a(t_1)-\rho h_a(t_1)h_b(t_2)), \\
c_2=c_2(\rho, t_1,t_2)&=&-t_2(1+\rho h_b(t_2)-\rho h_a(t_1)h_b(t_2)), \\
c_3=c_3(\rho, t_1,t_2)&=&t_1t_2(1-\rho h_a(t_1)h_b(t_2)).
\end{eqnarray*}

\item Further assume $\| \rho a(t_1)b(t_2)\| <1$. Then $c_0+c_1a+c_2b+c_3ab$ is invertible and
$$
\varphi((c_0+c_1a+c_2b+c_3ab)^{-1})=\frac{\tilde{h}_a(t_1)\tilde{h}_b(t_2)}{1-\rho [\tilde{h}_a(t_1)-h_a(t_1)][\tilde{h}_b(t_1)-h_b(t_2)]}.
$$

\end{enumerate}

\end{lemma}

\begin{proof}

For the first part, notice that $a(t_1)$  and $b(t_2)$ lie in the unital $C^*$-algebras generated by $a$ and $b$, respectively, hence these elements are      c. free.  On the other hand, it is  straightforward to prove that $\psi(a(t_1))= \psi(b(t_2)   =0$.

Part (2) is a direct algebraic computation as in Lemma 3.3. in  \cite{Haa}.

For part (3), 
 notice that Item $(1)$ implies that  
 $   \varphi(a(t_1)) \varphi(b(t_2)) =  \varphi(a(t_1) b(t_2))   $ and, therefore,   $|\rho \varphi(a(t_1)) \varphi(b(t_2))|\leq \| \rho a(t_1)b(t_2) \|  <1$. It follows that $1-   \rho \varphi(a(t_1))\varphi(b(t_2)) $ and $1-\rho a(t_1)b(t_2)$ are invertible. Then we can write
\begin{eqnarray*}
(c_0+c_1a+c_2b+c_3ab)^{-1}&=&  (1-t_2b)^{-1}(1-\rho a(t_1)b(t_2))^{-1}(1-t_1a)^{-1} \\
&=& (h_b(t_2)1_\mathcal{A}+b(t_2))(1-\rho a(t_1)b(t_2))^{-1}\\
& & (h_a(t_1)1_\mathcal{A}+a(t_1))\\
&=& f_0+f_1+f_2+f_3,
\end{eqnarray*}
where:
\begin{eqnarray*}
f_0&=&h_a(t_1)h_b(t_2) \sum_{n=0}^\infty \rho^n (a(t_1)b(t_2))^n, \\
f_1&=&h_b(t_2)\sum_{n=0}^\infty \rho^n (a(t_1)b(t_2))^n a(t_1 ), \\
f_2&=&h_a(t_1)\sum_{n=0}^\infty \rho^n b(t_2) (a(t_1)b(t_2))^n, \\
f_3&=&\sum_{n=0}^\infty \rho^n b(t_2)(a(t_1)b(t_2))^n a(t_1).
\end{eqnarray*}

Since $\psi(a(t_1))=\psi(b(t_2))=0$ and $a(t_1),b(t_2)$ are  c. free  w.r.t. $(\varphi, \psi )$ we obtain
$$
\varphi \left( \sum_{n=0}^\infty  ( \rho a(t_1)b(t_2))^n\right)=\sum_{n=0}^\infty \rho^n \varphi(a(t_1))^n \varphi(b(t_2))^n=(1-   \rho\varphi(a(t_1))\varphi(a(t_2)))^{-1}
$$
By the definition of $\tilde{h}_a$ and $\tilde{h}_b $ and Eq. \eqref{ab}, 
$$
D:=\varphi(a(t_1))\varphi(b(t_2))=[\tilde{h}_a(t_1)-h_a(t_1)]
[ \tilde{h}_b(t_2)  -h_b(t_2)].
$$

It follows that (here we use again that $\psi(a(t_1))=\psi(b(t_2))=0$ and $a(t_1),b(t_2)$ are    c. free   w.r.t. $(\varphi, \psi )$),

\begin{eqnarray*}
\varphi(f_0)&=&h_a(t_1)h_b(t_2)(1-\rho D)^{-1}, \\
\varphi(f_1)&=&h_b(t_2)[\tilde{h}_a(t_1)-h_a(t_1)](1-\rho D)^{-1},\\
\varphi(f_2)&=&h_a(t_1)[\tilde{h}_b(t_2)-h_b(t_2)](1-\rho D)^{-1},\\
\varphi(f_3)&=&  D(1-\rho D)^{-1}.
\end{eqnarray*}

A direct computation shows
\begin{eqnarray*}
\tilde{h}_a(t_1)\tilde{h}_b(t_2) &=& h_a(t_1)h_b(t_2)+  
[\tilde{h}_a(t_1)-h_a(t_1)]h_b(t_2) \\  
&+&  h_a(t_1)[\tilde{h}_b(t_2)-h_b(t_2)]\\
&+&    [\tilde{h}_a(t_1)-h_a(t_1)]  [\tilde{h}_b(t_2)-h_b(t_2)] .
\end{eqnarray*}

\end{proof}

A minor modification of the proof of  Theorem 3.4  in \cite{Haa}  proves the next lemma.

\begin{lemma}\label{Lemma:IdentityTildeha+b}
Assume $a$ and $b$ are  c. free with respect to $(\varphi,\psi)$ and $|t_1|<\frac{1}{4\|a\|}, |t_2|< \frac{1}{4\|b\|}$.

Then $h_a(t_1)\neq0, h_b(t_2)\neq0$, $h_a(t_1)+h_b(t_2) -1 \neq0$.

Even more, if 
\begin{align}\label{Hypo}
t_1h_a(t_1)=t_2h_b(t_2) 
\end{align} 
then

\begin{enumerate}
\item the number 
$$t:=\frac{t_1h_a(t_1)}{h_a(t_1)+h_b(t_2)-1}=\frac{t_2h_b(t_2)}{h_a(t_1)+h_b(t_2)-1},$$
satisfies $|t|<\frac{1}{\|a+b\|}.$
\item for $t$ as before it also holds
$$
\tilde{h}_{a+b}(t)=[h_a(t_1)+h_b(t_2)-1] \frac{\tilde{h}_a(t_1)\tilde{h}_b(t_2)}{h_a(t_1)\tilde{h}_b(t_2)+   h_b(t_2)\tilde{h}_a(t_1) -\tilde{h}_a(t_1)\tilde{h}_b(t_2)}.
$$
\end{enumerate}

\end{lemma}

\begin{proof}
First part is just Lemma  is part (a) of Theorem 3.4 in \cite{Haa}.  

Regarding part (2), taking $\rho=\frac{1}{h_a(t_1)h_b(t_2)}$ in previos lemma yileds
\begin{eqnarray*}
c_0&=&\frac{h_a(t_1)+h_a(t_2)-1}{h_a(t_1)h_b(t_2)}, \\
c_1&=&-\frac{t_1}{h_b(t_2)}, \\
c_2&=&-\frac{t_2}{h_a(t_1)},\\
c_3&=&0,\\
t&=&-\frac{c_1}{c_0}   = -\frac{c_2}{c_0},
\end{eqnarray*}
and $c_0+c_1a+c_2b+c_3ab=c_0(1-t(a+b))$.

Hence
\begin{eqnarray*}
\tilde{h}_{a+b}(t)&=& c_0\varphi((c_0+c_1a+c_2b+c_3ab)^{-1}) \\
&=&c_0\tilde{h}_a(t_1)\tilde{h}_b(t_2)(1-\rho [\tilde{h}_a(t_1)-h_a(t_1)][\tilde{h}_b(t_2)-h_b(t_2)])^{-1}
\end{eqnarray*}

which simplifies to the desired expression.
\end{proof}

Now we are in position to prove the desired linearizing property.
\begin{thm}
Let $a$  and $b$ be to elements in  $\mathcal{A}$  c. free   with respect to $(\varphi, \psi)$.  For $|z| < \min\{ \frac{1}{6\|a\|}, \frac{1}{6\|b\|}, \frac{1}{6\|a+b\|} \}$,
\begin{equation}\label{Eqn:LineariztionPropertycR}
R_{a+b}^{(\varphi,\psi)}(z)=R_a^{(\varphi,\psi)}(z)+R_b^{(\varphi,\psi)}(z). 
\end{equation}
\end{thm} 

\begin{proof}

Firstly, notice that if $\varphi$ is a $*$-homomorphism, $R_{a}^{(\varphi,\psi)}(z)=\varphi(a)$ for any $a$. Hence it is trivial that $R_{a+b}^{(\varphi,\psi)}=R_{a}^{(\varphi,\psi)}+R_{b}^{(\varphi.\psi)}$. Hence we might assume $\varphi$ is not a $*$-homomorphism.

By Theorem \ref{Thm:FreeCopiesCFAlgebras} we might assume that $a$ and $b$ are free with respect to $\psi$.

Since $|z|< \min\{ \frac{1}{6\|a\| } , \frac{1}{6\|b\|}, \frac{1}{6\|a+b\|} \}$,  all inverse functions  nedded to define
 $R^{(\varphi,\psi)}$  are well defined. Now, take
\begin{eqnarray*}
t_1=k_a^{-1}(z), t_2=k_b^{-1}(z), t_3=k_{a+b}^{-1}(z)
\end{eqnarray*}
that satisfy Eq. \eqref{Hypo}:
\begin{align} \label{Hypoo}
t_1 h_a(t_1) = k_{a}(t_1)= z = k_b(t_2)= t_2 h_b(t_2).
\end{align}

With this notation, and  by Definition \ref{Def:cRCstarAlg} (Eq.   \eqref{Eqn:cR-T}), we can rewrite \eqref{Eqn:LineariztionPropertycR}  as
\begin{equation}\label{Eqn:ha+b}
\frac{1}{t_3}\left( 1-\frac{1}{\tilde{h}_{a+b}(t_3)}  \right)=\frac{1}{t_1}\left( 1 - \frac{1}{\tilde{h}_a(t_1)}  \right)+\frac{1}{t_2}\left( 1 - \frac{1}{\tilde{h}_b(t_2)}  \right)
\end{equation}

 Lemma \ref{Lemma:IdentityTildeha+b} (we can use it because Eq. \eqref{Hypoo} is satisfied) implies that the number (see also Eq. \eqref{Hypoo}, again)
 
\begin{align}\label{t}
t=\frac{t_1h_a(t_1)}{h_a(t_1)+h_b(t_2)-1} = \frac{t_2h_b(t_2)}{h_a(t_1)+h_b(t_2)-1}   ,
\end{align}
satisfies
\begin{align*}
\tilde{h}_{a+b}(t)=[h_a(t_1)+h_b(t_2)-1]\frac{\tilde{h}_a(t_1)\tilde{h}_b(t_2)}{h_a(t_1)\tilde{h}_b(t_2)+h_b(t_2)\tilde{h}_a(t_1)-\tilde{h}_a(t_1)\tilde{h}_b(t_2)}.
\end{align*}
From the latter we obtain 
\begin{eqnarray}
\frac{1}{t} \left( 1-\frac{1}{\tilde{h}_{a+b}(t)} \right)&=& \frac{\tilde{h}_{a+b}(t)-1}{t\tilde{h}_{a+b}(t)} \nonumber \\
&=& \frac{h_a(t_1)\tilde{h}_a(t_1)\tilde{h}_b(t_2)}{t(h_a(t_1)+h_b(t_2)-1)\tilde{h}_a(t_1)\tilde{h}_b(t_2)} \nonumber\\
& +& \frac{h_b(t_2)\tilde{h}_a(t_1)\tilde{h}_b(t_2)}{t(h_a(t_1)+h_b(t_2)-1)\tilde{h}_a(t_1)\tilde{h}_b(t_2)} \nonumber\\
&- & \frac{h_a(t_1)\tilde{h}_b(t_2)}{t(h_a(t_1)+h_b(t_2)-1)\tilde{h}_a(t_1)\tilde{h}_b(t_2)}\nonumber \\
&-& \frac{h_b(t_2)\tilde{h}_a(t_1)}{t(h_a(t_1)+h_b(t_2)-1)\tilde{h}_a(t_1)\tilde{h}_b(t_2)}   \nonumber  \label{Eqn:Tildeha+b}
\end{eqnarray}

Using Eq. \eqref{t} (the term in the middle) we obtain that
\begin{align}\label{pin1}
\frac{h_a(t_1)\tilde{h}_a(t_1)\tilde{h}_b(t_2)}{t(h_a(t_1)+h_b(t_2)-1)\tilde{h}_a(t_1)\tilde{h}_b(t_2)}  - & \frac{h_a(t_1)\tilde{h}_b(t_2)}{t(h_a(t_1)+h_b(t_2)-1)\tilde{h}_a(t_1)\tilde{h}_b(t_2)} \notag
\\ 
&  \hspace{1.5cm} =   \frac{1}{t_1}\left( 1 -\frac{1}{\tilde{h}_a(t_1)} \right).
\end{align}   
 
Using Eq. \eqref{t} (the term on the right) we obtain that
\begin{align}\label{pin2}
\frac{h_b(t_2)\tilde{h}_a(t_1)\tilde{h}_b(t_2)}{t(h_a(t_1)+h_b(t_2)-1)\tilde{h}_a(t_1)\tilde{h}_b(t_2)} - & \frac{h_b(t_2)\tilde{h}_a(t_1)}{t(h_a(t_1)+h_b(t_2)-1)\tilde{h}_a(t_1)\tilde{h}_b(t_2)}
 \notag
\\ 
&  \hspace{1.5cm} =   \frac{1}{t_2}\left( 1 -\frac{1}{\tilde{h}_b(t_2)} \right) .
\end{align}   
Eqs. \eqref{Eqn:Tildeha+b}, \eqref{pin1} and \eqref{pin2} imply that
\begin{align}\label{pin3}
\frac{1}{t} \left( 1-\frac{1}{\tilde{h}_{a+b}(t)} \right) = \frac{1}{t_1}\left( 1 -\frac{1}{\tilde{h}_a(t_1)} \right)+\frac{1}{t_2}\left( 1 -\frac{1}{\tilde{h}_b(t_2)} \right). 
\end{align}

Comparing \eqref{Eqn:ha+b} and \eqref{pin3}, it suffices to show
$$
\frac{1}{t_3} \left( 1-\frac{1}{\tilde{h}_{a+b}(t_3)} \right)=\frac{1}{t} \left( 1-\frac{1}{\tilde{h}_{a+b}(t)} \right).
$$
But, since $a$ and $b$ are free with respect to $\psi$, Theorem 4.3 in \cite{Haa} implies $t=t_3$, and hence we get the conclusion.
\end{proof}

\end{document}